\documentclass[12pt]{article}
\usepackage{amsthm,bm,amssymb}
\usepackage{color}
\usepackage{graphicx}
\usepackage{fancyhdr}
\usepackage{verbatim}
\usepackage{mathrsfs}
\usepackage{amsfonts,amscd,amssymb,epsfig,latexsym,amsmath,enumerate}
\usepackage{mathtools}
\usepackage{verbatimbox}
\usepackage{tikz}
\usepackage{tikz-cd}
\usepackage{bbding}
\usepackage{dsfont}
\usepackage{slashed}

\usepackage{yfonts}

\usepackage{graphicx}
\usepackage[utf8]{inputenc}
\usepackage[export]{adjustbox}
\usepackage{wrapfig}

\graphicspath{ {images/} } 
\usepackage{float}

\usetikzlibrary{matrix,arrows,decorations.pathmorphing}
\usepackage{enumerate}
\usepackage [english]{babel}
\usepackage [autostyle, english = american]{csquotes}
\MakeOuterQuote{"}

\newtheorem{theorem}{Theorem}[section]
\newtheorem{lemma}[theorem]{Lemma}
\newtheorem{proposition}[theorem]{Proposition}
\newtheorem{corollary}[theorem]{Corollary}

\newtheorem{definition}[theorem]{Definition}
\newtheorem{ex}[theorem]{Example}
\newenvironment{example}{\begin{ex} \rm}{\end{ex}}

\newtheorem{rmk}[theorem]{Remark}
\newenvironment{remark}{\begin{rmk} \rm}{\end{rmk}}
\newtheorem{ques}[theorem]{Question}
\newenvironment{question}{\begin{ques} \rm}{\end{ques}}

\renewcommand{\caption}{}
\numberwithin{equation}{section}
\def\&{\wedge}

\newcommand{\bb}{\mathbb}

\newcommand{\R}{\bb{R}}
\newcommand{\C}{\bb{C}}

\newcommand{\Z}{\bb{Z}}

\DeclareMathOperator{\id}{id}

\usepackage{geometry}
\title{Adjacent Singularities, TQFTs, and Zariski's Multiplicity Conjecture}
\author{Shamuel Auyeung}

\begin{document}
	\maketitle
	
	\begin{abstract}
		We give a new proof of Zariski's multiplicity conjecture in the case of isolated hypersurface singularities; this was first proved by de Bobadilla-Pe\l ka \cite{BobadillaPelka}. Our proof uses the TQFT structure of fixed-point Floer cohomology and the fact that adjacent singularities produce symplectic cobordisms between the Milnor fibrations of the singularities. The key technical result is to construct a chain map on Floer cochains using the cobordism and as a last step, apply a spectral sequence of McLean \cite{McLeanLog}. This last step allows us to also recover a theorem of Varchenko \cite{varchenko}.
	\end{abstract}
	
	\begin{center}
	\textit{Dedicated to Philip and Marie Holtrop.}
	\end{center}
	
	\tableofcontents
		
	\section{Introduction}
	
	Let $f \in \C[z_1,...,z_{n+1}]$ be a polynomial such that $f(0) = 0$ and 0 is an isolated singularity. This means that all the first partial derivatives of $f$ with respect to the standard complex coordinates vanish at 0 and there exists a neighborhood of 0 in $V:= V(f) = f^{-1}(0)$ such that there are no singularities in the neighborhood other than 0. We will often refer to $f$ as the singularity though to be precise, we are really considering the germ of 0 in $f^{-1}(0)$. To such a singularity, we may assign the \textbf{multiplicity} $\nu(f)$ which is the largest integer such that $f \in \mathfrak{m}^{\nu(f)}$; here $\mathfrak{m}$ is the maximal ideal of $\C[z_1,...,z_{n+1}]$. We may also associate to the singularity its Milnor number $\mu(f):= \dim_\C \C[z_1,...,z_{n+1}]/\text{Jac}(f)$ where $\text{Jac}(f)$ is the Jacobian ideal of $f$.
	
We can be somewhat more general than our current set up and consider the formal power series ring instead of the polynomial ring. Under this more general set up, de Bobadilla-Pe\l ka \cite{BobadillaPelka} proved the following :

\begin{theorem}
Let $(f_t)$ be a continuous family of power series. If the Milnor number $\mu(f_t)$ is independent of $t$ and finite then the multiplicity $\nu(f_t)$ is also independent of $t$.
\end{theorem}

As they explain, this gives a positive answer to Zariski's multiplicity conjecture (ZMC) for the case of families of isolated hypersurface (holomorphic) singularities; their paper outlines a history of the conjecture. Their proof begins with constructing an A'Campo space of ``radius zero'' monodromy, making use of ideas from tropical geometry. They then endow the space with a symplectic structure and extend a spectral sequence of McLean \cite{McLeanLog} in order to accommodate the zero radius monodromy. As a last step, a theorem of L\^e-Ramanujam \cite{le_ramanujam} guarantees that a piece of the $E^1$-page is in fact, trivial. This triviality ensures the constancy of the multiplicity. Though they don't state it, their proof should also recover a theorem of Varchenko \cite{varchenko} regarding log canonical thresholds.

The $E^1$ page of the McLean spectral sequence takes data of the exceptional divisors of a log resolution of the singularity and converges to a symplectic invariant known as fixed-point Floer cohomology $HF^*$. This Floer invariant is defined for the monodromy map associated to an isolated hypersurface singularity. The spectral sequence is the key tool McLean used to prove that the $HF^*$ of iterates of the monodromy recovers multiplicity and log canonical threshold.

Since de Bobadilla-Pe\l ka extend a spectral sequence related to symplectic geometry \cite{McLeanLog}, one may naturally ask 

\begin{question}
Is there a proof of ZMC that is more symplectic-geometric in flavor, perhaps one that takes advantage of the additional structures of $HF^*$? If so, the study of the additional structures may be a road into results beyond ZMC.
\end{question}

The purpose of this paper is to supply such a proof. We've elected to work with polynomials rather than power series mostly for the sake of simplifying exposition but so long as the Milnor numbers we work with are finite, the theorem we prove extends to holomorphic germs of singularities:

	\begin{theorem} \label{zar}
		Suppose a family of isolated hypersurface singularities where its members are adjacent all have the same Milnor number $\mu$. Then each member in the family also has the same multiplicity and log canonical threshold.
	\end{theorem}

In broad strokes, our proof has some similarities with \cite{BobadillaPelka} in that we work with the original McLean spectral sequence and also use the result of  L\^e-Ramanujam. However, rather than extend the spectral sequence, we make use of the TQFT structure of fixed-point Floer cohomology; i.e. symplectic cobordisms induce maps on the level of $HF^*$. Moreover, when two singularities are adjacent, there are natural cobordisms between the mapping tori of iterated monodromies $\phi^k$ coming from the singularities. These are utilized for establishing an exact triangle as in Lemma \ref{triangle} and some Picard-Lefschetz theory plus the result of L\^e-Ramanujam show that two relevant fixed-point Floer cohomologies are isomorphic. Then, applying McLean's original spectral sequence, we obtain a positive answer to ZMC for isolated hypersurface singularities.

	\subsection{Organization}
	
	In Section \ref{alg}, we introduce the relevant algebraic invariants which fixed-point Floer cohomology $HF^*$ recovers. Section \ref{floer} is where the definition of $HF^*$ is reviewed. Secion \ref{adj} introduces the notion of adjacent singularities and constructs cobordisms between the mapping tori of monodromies arising from an adjacent pair. The cobordism induces a map which is the subject of Section \ref{map}. Sections \ref{ham} and \ref{cpt} give the needed notions for showing why the induced map is well-defined. Per usual in Floer theory, we prove a version of Gromov compactness via a neck lengthening argument (this is different from SFT neck stretching). Section \ref{tqft} gives further language of TQFTs and its main result shows that our induced map is dependent only on the link of the singularities and the canonical framing. In the process, we use so-called partially-defined Lefschetz fibrations which may be of independent interest. We do not claim any originality in using them but it does seem that they've only appeared in unpublished literature. Section \ref{sunflower} gives a symplectically flavored proof of ZMC and the last section is more colloquial, posing questions and examples of various phenomena. In particular, since the Floer maps induced by the cobordisms can be defined for adjacent singularities when $\mu$ is not constant, a natural follow-up question is how they might be used to detect changes in multiplicity within families of isolated hypersurface singularities. Another natural TQFT question is whether the products $HF^*(\phi^k) \otimes HF^*(\phi^\ell) \to HF^*(\phi^{k+\ell})$ or counits $HF^*(\phi^k) \to \Z$ shed light on the nature of singularities. The appendices address gradings and some examples of families of isolated singularities.
		
	\section{Some Algebraic Invariants and Open Book Decompositions} \label{alg}
	
	Let $f:\C^{n+1} \to \C$ be a polynomial with isolated singularity and let $\C[z_1,...,z_{n+1}]$ be the ring of formal power series in $n$ variables. For convenience, we often assume $f(0)=0$. The maximal ideal $\mathfrak{m}$ consists of all polynomials without a constant term. Here are three algebraic invariants one can define for the germ of the singularity $(f^{-1}(0),0)$; we've seen the latter two already. The first involves integration but in fact, can be rendered in a completely algebro-geometric way.
	
	\begin{definition}
		Let $\phi_{c,f}:\mathbb{C}^n \setminus \{0\} \to \mathbb{C}$ be $1/|f|^{2c}$. The \textbf{log canonical threshold} of $f^{-1}(0)$ is $\text{lct}(f^{-1}(0)) = \sup\{c >0:\phi_{c,f} \text{ is locally integrable} \}$. Other common notation is $\text{lct}(f)$.
	\end{definition}
	
	The next invariant involves differentiation instead of integration.
	
	\begin{definition}
		The \textbf{multiplicity} of $f$, denoted $\nu(f)$ is the largest positive integer such that $f \in \mathfrak{m}^{\nu(f)}$.	
	\end{definition}
	
	\begin{definition}
		The \textbf{Milnor number} of $f$, denoted $\mu(f)$, is the complex dimension of the algebra $\C[z_1,...,z_{n+1}]/\text{Jac}(f)$ where $\text{Jac}(f)$ is the Jacobian ideal of $f$. This ideal is generated by the first partial derivatives of $f$.
	\end{definition}
	
	Though the Milnor number defined in this way is completely algebraic, there is another notion which Milnor defined in \cite{Milnor}. In fact, there are several ways other ways to think of $\mu(f)$. For sufficiently small $\epsilon \gg \delta >0$, the \textbf{Milnor fiber} is $M_f:= f^{-1}(\delta) \cap B_\epsilon(0)$ where $B_\epsilon(0) \subset \C^{n+1}$ is a radius $\epsilon$ ball centered at 0. We now have the following theorem:
	
	\begin{theorem}\cite{Milnor}
		Let $f:(\C^{n+1},0) \to (\C,0)$ be a polynomial with isolated singularity at 0 and $\mu(f)$ be the Milnor number defined above. Then:
		\begin{enumerate}
			\item  The Poincar\'e-Hopf index of the vector field $\nabla f$ at 0 (under the standard metric of $\C^{n+1}$) is $\mu(f)$.	
			\item 	The homotopy type of the Milnor fiber $M_f$ is a bouquet of spheres $\bigvee^{\mu(f)} S^n$. In particular, the Betti number $b_n(M_f) = \mu(f)$.
			\item If $\tilde{f}$ is a Morsification of $f$, then the number of critical points of $\tilde{f}$ is $\mu(f)$.
		\end{enumerate}
	\end{theorem}
	
	Here, a Morsification of a function $f$ is a $C^\infty$-small perturbation so that the resulting $\tilde{f}$ has only Morse critical points (this will be relevant in the final steps for proving ZMC). This theorem shows that the algebraically defined invariant $\mu(f)$ is also a topological invariant since the Poincar\'e-Hopf index, middle Betti number, and Morsifications are all topological. In the introduction, we also described the result by Milnor concerning the fibration structure. We'll restate the result here:
	
	\begin{theorem}\cite{Milnor}
		Let $f:(\C^{n+1},0) \to (\C,0)$ be a polynomial with isolated singularity at 0. For a sufficiently small $\epsilon > 0$, let $L := S^{2n+1}_\epsilon \cap f^{-1}(0)$. Then the map $\frac{f}{|f|}:S^{2n+1}_\epsilon \setminus L \to S^1$ is a fibration whose fibers are diffeomorphic to the interior of $M_f$, previously defined.	
	\end{theorem}
	
	In fact, much more is true. We observe that $M_f$, by virtue of being holomorphically embedded in $\C^{n+1}$ is also exact symplectically embedded. The parallel transport map along the base circle actually gives a symplectomorphism $\phi:M_f \to M_f$ supported away from $\partial M_f \cong L$ which is called the \textbf{monodromy map}. Hence, we form the mapping torus from the interior of $M_f$: $T_\phi:=\mathring{M}_f \times [0,1]/\sim$ where $(x,0) \sim (\phi(x),1)$. If we only consider the smooth structures, we recover Milnor's fibration. This leads us to a definition:
	
	\begin{definition}
		An \textbf{open book decomposition} $(M,\pi,B)$ is the data of:
		\begin{itemize}
			\item a binding $B \subset M$ which is a codim 2 smooth submanifold;
			\item a fibration $\pi:M \setminus B \to S^1$.
		\end{itemize}
	\end{definition}
	
	Hence, this mapping torus fibration is an example of an \textbf{open book}. But we see that the essential data towards defining it was an open exact symplectic manifold $\mathring{M}$ with a compactly supported symplectomorphism $\phi:\mathring{M} \to \mathring{M}$ from which we can construct the mapping torus $T_\phi$. And the binding is $\partial M$. We could just use this data to define a mapping torus with $\partial M$ as the binding. Hence, the data $(M,\phi)$ is called an \textbf{abstract contact open book}. To be more precise, we don't actually need any symplectic structure; it makes sense to define abstract open books with just diffeomorphisms and smooth pages.
	
	\begin{example}
		Here is a classic and overdue example. Let $f=x^2+y^3$ and $g=x^2+y^2$. Then $\mathfrak{m}^2$ contains elements of degree at least 2: $x^2,y^2,xy,x^3,...$ and $\mathfrak{m}^3$ contains elements of degree at least 3. So the multiplicities are $\nu(f) = \nu(g) = 2$. It is clear that $\mu(f)=2$ and $\mu(g)=1$. Using some additional results from \cite{Milnor} for counting the number of boundary components, since the Milnor fibers are complex curves, we're able to deduce that $M_g$ is a genus 0 curve with two boundary components whereas $M_f$ is a genus 1 curve with one boundary component. The log canonical threshold of the singularity is $5/6$; a computation can be found in \cite{posII}.
	\end{example}

\subsection{Supporting Open Book Decompositions of Contact Manifolds}

	However, it turns out that when we take into account the symplectic structures, the total space of the mapping torus defined from an abstract open book admits a contact structure $\xi:=\ker \alpha$ (for some contact 1-form $\alpha$) via the Thurston-Winkelnkemper construction that is compatible with the open book structure. What this means is that $B$ is transverse to $\xi$ and the Reeb field of $\alpha$ is always transverse to the pages. This perspective begins with a symplectic discussion but we could also go the other way and begin with the contact form.
	
	\begin{definition}
		A \textbf{supporting open book decomposition} of a contact manifold $(Y,\xi)$ is
		\begin{itemize}
			\item $B \subset Y$, a codim 2 submanifold that is transverse to $\xi$;
			\item a fibration $Y \setminus B \to S^1$
			\item there is a Reeb vector field that is always transverse to the fibers (pages).
		\end{itemize}
	\end{definition}
	
	From this point of view which does not begin with symplectic structures, the idea is that, as much as possible, we want the pages to be tangent to the contact structure which means we get a symplectic form on the pages.
	
	Having introduced some objects from symplectic and contact geometry, we'll now introduce the main result of \cite{McLeanLog} which relates these with some of the algebraic invariants we saw earlier:
	
	\begin{theorem}
		Let $f,g:\C^{n+1} \to \C$ be polynomials with isolated singularities at 0. Then, if the links of $f$ and $g$ are (graded) contactomorphic, the multiplicity and log canonical threshold of $f$ and $g$ are equal.
	\end{theorem}
	
	The key technical tool is a spectral sequence relating the exceptional divisors of a log resolution and the fixed-point Floer cohomology of the monodromy which we reproduce here (with slight rephrasing) though we will not explain the technical aspects of it unless needed. See the paper for details.

\begin{theorem}[Theorem 1.2 of \cite{McLeanLog}]
Suppose that $\pi:Y \to \C^{n+1}$ is a multiplicity $m$ separating resolution for some $m \in \Z_+$ and $\check{S}$ an indexing set of the exceptional divisors. Let $(w_i)_{i \in \check{S}}$ be positive integers so that $-\sum_{i \in \check{S}} w_iE_i$ is ample. Let $a_i$ be the discrepancy of $E_i$, $m_i = \text{ord}_f(E_i)$, and $k_i:=m/m_i$ for all $i \in S_m$ which is the set of $i$ where $m_i$ divides $m$. Then there is a cohomological spectral sequence converging to the fixed-point Floer cohomology $HF^*(\phi^m,+)$ of the $m$th iterate of the monodromy $\phi$ with $E_1$ page
\[E^{p,q}_1= \bigoplus_{i \in S_m; \, k_i w_i=-p} H_{n-(p+q)-2k_i(a_i+1)} (\widetilde{E}^o_i;\Z)  \]

\noindent where $\widetilde{E}^o_i$ is a particular $m_i$-fold cover of $E^o_i:= E_i \setminus \bigcup_{j \neq i} E_j$.

\end{theorem}	
	
\noindent The spectral sequence can be viewed as a categorification of a result of \cite{acampo} who proved the the classical Lefschetz number $\Lambda(\phi^m) = \sum_{i \in S_m} m_i \chi(E^0_i)$ for all $m > 0$. That is, $\Lambda(\phi^m)$ can be computed from the data of \textit{any} log resolution of the singularity. McLean's result is a categorification in the sense that this numerical statement is lifted to the level of homological algebra (when we view a spectral sequence as implied by a filtered complex) and recovered when we take the Euler characteristic. This is because the Euler characteristic of a spectral sequence equals the Euler characteristic of what it converges to and $\chi(HF^*(\phi^m,+) = \Lambda(\phi^m)$. Given that this Floer invariant gives us a way to break into the algebro-geometric setting, we now give a quick recap of what the invariant is and its properties.
	
	\section{Review of Fixed-Point Floer Cohomology} \label{floer}
	
	We follow Mclean \cite{McLeanLog} but the original definition can also be found in \cite{SeidelMut}. Let $(M, \theta)$ be a Liouville domain. An almost complex structure $J$ on $M$
	is \textbf{cylindrical} near $\partial M$ if it is compatible with the symplectic form $d\theta$ and if $dr \circ J = -\alpha$ near $\partial M$ inside the standard collar neighborhood $(0, 1] \times \partial M$ where $\alpha= \theta|_{\partial M}$.
	
	We will consider compactly supported exact symplectomorphisms $\phi: M \to M$ because then fixed point Floer cohomology will have finite rank. Such a $\phi$ is \textbf{non-degenerate} if for every fixed point $p$ of $\phi$ the linearization of $\phi$ at $p$ does not have 1 as an eigenvalue. It has \textbf{small positive slope} if it is equal to the time-1 Hamiltonian flow of $\delta r$ near $\partial M$ where $\delta > 0$ is smaller than the period of the smallest periodic Reeb orbit of $\alpha$. This means that it corresponds to the time $\delta$ Reeb flow near $\partial M$. If $\phi$ is an exact symplectomorphism, then a small positive slope perturbation $\check{\phi}$ of $\phi$ is an exact symplectomorphism $\check{\phi}$ equal to the composition of $\phi$ with a $C^\infty$-small Hamiltonian symplectomorphism of small positive slope. The \textbf{action} of a fixed point $p$ is $-F_\phi(p)$ where $F_\phi$ is a function satisfying $\phi^* \theta = \theta + dF_\phi$. The action depends on a choice of $F_\phi$ which has to be fixed when $\phi$ is defined although usually $F_\phi$ is chosen so that it is zero near $\partial M$ (if possible). All of the symplectomorphisms coming from isolated hypersurface singularities will have such a unique $F_\phi$. An isolated family of fixed points is a path connected compact subset $B \subset M$ consisting of fixed points of $\phi$ of the same action and for which there is a neighborhood $N \supset B$ where $N \setminus B$ has no fixed points. Such an isolated family of fixed points is called a \textbf{codimension 0 family of fixed points} if in addition there is an autonomous Hamiltonian $H: N \to (-\infty, 0]$ so that $H^{-1}(0) = B$ is a connected codimension 0 submanifold with boundary and corners, the time $t$ flow of $X_H$ is well-defined for all $t \in \R$ and $\phi|_N: N \to N$ is equal to the time 1 flow of H. The action of an isolated family of fixed points $B \subset M$ is the action of any point $p \in B$.
	
	Let $(M, \theta, \phi)$ be an abstract contact open book. Let $(J_t)_{t\in [0,1]}$ be a smooth
	family of almost complex structures with the property that $\phi^*J_0 = J_1$. A Floer trajectory of $(\phi, J_t)_{t\in [0,1]})$ joining $p_-, p_+ \in M$ is a smooth map $u:\R \times [0,1] \to M$ so that $\partial_s u + J_t \partial_t u = 0$ where $(s,t)$ parameterizes $\R \times \R/\Z$, $u(s,0) = \phi(u(s,1))$ and so that $\lim_{s\to \pm \infty} u(s,t) = p_\pm$ for all $t \in [0,1]$. We write $\mathcal{M}(\phi,J_t, p_-, p_+)$ for the set of such Floer trajectories and define $\overline{\mathcal{M}}(\phi,J_t, p_-, p_+) := \mathcal{M}(\phi,J_t, p_-, p_+)/\R$, where $\R$ acts by translation in the $s$ coordinate.
	
	In order to properly define the fixed point Floer cohomology of an exact symplectomorphism, we need to discuss gradings. However, we move that discussion to the appendix.
	
	Let $(M, \theta, \phi)$ be a graded abstract contact open book and let $\check{\phi}$ be a small positive slope perturbation of $\phi$. This can be done so that $\check{\phi}$ is $C^\infty$-close to $\phi$ and so that the fixed points of $\check{\phi}$ are non-degenerate. We can also ensure that $\check{\phi}$ is a graded symplectomorphism due to the fact that it is isotopic to $\phi$ through symplectomorphisms. We now choose a $C^\infty$-generic family of cylindrical almost complex structures $(J_t)_{t\in[0,1]}$ satisfying $\phi^*J_0 = J_1$. The genericity property then tells us that $\overline{\mathcal{M}}(\phi,J_t, p_-, p_+)$ is a compact
	oriented manifold of dimension 0 for all fixed points $p_-, p_+$ of $\phi$ satisfying $CZ(p_-)-CZ(p_+) = +1$. 
	
	We define $\#^\pm \overline{\mathcal{M}}(\phi,J_t, p_-, p_+)$ to be the signed count of elements
	of $\overline{\mathcal{M}}(\phi,J_t, p_-, p_+)$. Let $CF^*(\check{\phi})$ be the free abelian group generated by fixed points of $\phi$ and graded by the Conley-Zehnder index taken with \textbf{negative} sign. The differential $\partial$ on $CF^*(\check{\phi})$ is a $\Z$-linear map satisfying $\partial (p_+) = \sum_{p_-} \#^\pm \overline{\mathcal{M}}(\phi,J_t, p_-, p_+) \cdot p_-$ for all fixed points $p_+$ of $\check{\phi}$ where the sum is over all fixed points $p_-$ satisfying $(-CZ(p_-))-(-CZ(p_+)) =+1$. Since we use $-F_\phi(p)$ for the action, the differential takes fixed points of lower action to fixed points of higher action.
	
	Because $(J_t)_{t\in [0,1]}$ is $C^\infty$-generic, one can show that $\partial^2 = 0$ and we define the resulting homology group to be $HF^*(\phi,+):= HF^*(\check{\phi},(J_t)_{t \in[0,1]})$. The notation is justified because this does not depend on the choice of perturbation $\check{\phi}$ nor on the choice of almost complex structure $(J_t)$. These conventions then tell us that if $\phi: M \to M$ is the identity map with the trivial grading and $\dim M = n$ then $HF^i(\phi, +) = H^{n+i}(M,\Z)$.
	
	Having defined fixed-point Floer cohomology, we list some of their properties which can be found in \cite{McLeanLog}.
	
	\begin{enumerate}
		\item For a graded abstract contact open book $(M, \lambda, \phi)$, the Lefschetz number $\Lambda(\phi)$ is equal to the Euler characteristic of $HF^*(\phi, +)$ multiplied by $(-1)^n$ where $n=\frac{1}{2}\dim M$.
		
		\item If $(M_j,\lambda_j,\phi_j)$, $j=0,1$ are graded abstract contact open books so
		that the graded contact pairs associated to them are graded contactomorphic, then $HF^*(\phi_0, +) \cong HF^*(\phi_1, +)$
		
		\item Let $(M,\lambda,\phi)$ be a graded abstract contact open book where $\dim M = 2n$. Suppose that the set of fixed points of a small positive slope perturbation $\check{\phi}$ of $\phi$ is a disjoint union of codimension 0 families of fixed points $B_1,...,B_\ell$ and $\iota:\{1,...,\ell\} \to \mathbb{N}$ is a function where $\iota(i) = \iota(j)$ if and only if the actions of $B_i$ and $B_j$ equal and $\iota(i) < \iota(j)$ if the action of $B_i$ is less than that of $B_j$. Then there is a cohomological spectra sequence converging to $HF^*(\phi,+)$ whose $E_1$ page is equal to
		
		\[E^{p,q}_1 = \bigoplus_{\iota(i)=p} H_{n-(p+q)-CZ(\phi,B_i)} (B_p,\Z). \]
	\end{enumerate}		
	
	The last property is not one that we'll use directly but it's important for McLean's result. Here are some brief comments. We can make $\iota$ easily using the actions on the $B_i$ and this gives a filtration on a chain complex built out of the fixed points. It's a general fact that whenever we have a filtered complex, we can obtain from it a spectral sequence.

\section{When $k=\nu$ is the Multiplicity} \label{mult}
	
In the previous two sections, we've discussed algebraic invariants of singularities and also fixed-point Floer cohomology. Here, we briefly mention some relationships between them since our main result is to use this Floer invariant to prove a version of ZMC concerned with the multiplicity of singularities. For convenience in writing grading shifts in a moment, suppose $f:\mathbb{C}^{n+1} \to \mathbb{C}$. If the multiplicity is $\nu$, then the McLean spectral sequence for $\phi^\nu$ degenerates at the $E^1$ page; i.e. $E^1 = E^\infty$. This is because for any log resolution, there is one divisor whose order of vanishing is the multiplicity and the order of vanishing of all the other divisors are larger. Since those do not divide $\nu$, they do not appear in the $E^1$ page and so there is only one entry in the $E^1$ page which means the spectral sequence degenerates immediately. Moreover, a $k$ separating resolution will give us trivial input data for the $E^1$ page when $k < \nu$. In this way, McLean's spectral sequence gives us a way to determine the multiplicity $\nu$ from fixed-point Floer cohomology; just look for the first nontrivial $HF^*(\phi^k,+)$.
	
It is shown in \cite{contactloci}, Prop 1.6, that if we write $f = f_\nu + f_{\nu+1}+...$ where each term $f_d$ is a homogeneous polynomial of degree $d$, then $HF^*(\phi^\nu) \cong H_c^{*+2nd+n-1}(\chi_\nu(f)) \cong H_c^*(F)$. Here, $\chi_\nu(f)$ is the $\nu$-th contact loci and $F$ is the Milnor fiber of $f_\nu$. This is also described on slide 14 of McLean's slides \cite{McLeanArc}.  \\

\noindent \textbf{Remark:} The authors of \cite{contactloci} prove the existence of a $k$-separating resolution because McLean's proof contains a mistake.
		
\section{Adjacent Singularities and a Cobordism} \label{adj}

%an adjacency preserves the multiplicity $m$, then the induced map on the $E^1$ pages is an isomorphism and also the induced map for $HF(\phi^m,+)$ and each is isomorphic to the Milnor fiber of the lowest homogeneous parts. For example, with $x^2+\epsilon y^2 + y^3$, the Milnor fiber of the lowest homogeneous part when $\epsilon = 0$, is just two points since we're looking at $x^2 = 1$. This is the double cover of a single point which is homotopy equivalent to $\mathbb{P}^1 \setminus 1 \, pt$. For $\epsilon > 0$, the Milnor fiber is $x^2+y^2 = 1$ which is $\mathbb{C}^* = \mathbb{P}^1 \setminus 2 \, pts$. The homology of 2 points and $\mathbb{C}^*$ are isomorphic as vector spaces but not as graded vector spaces.

\begin{definition}
	Two isolated hypersurface singularities are \textbf{adjacent} if they come from a family of polynomials $f_t:\C^{n+1} \to \C$, $t \in [0,1]$ where $f_0$ gives one of the singularities and $f_t$ for $t \neq 0$ gives the other. The notation is $[f_1] \to [f_0]$. Alternatively, there exists an arbitrarily small perturbation $p$ such that $[f_0 + p] = [f_1]$.
\end{definition}
	
Note that the relationship is not symmetric in the sense that there might not be a family where $f_1$ gives the central fiber and $f_0$ gives the rest. There is a more general definition for non-hypersurface singularities but this will suffice for us.
	
	\begin{example}
		Let $f = x^2 + y^3$ and $g = x^2+y^2(1+y)$. The singularity type of $g$ is $A_1$ and near 0, the singularity is given by $xy = 0$ and the Milnor fiber is isomorphic to $\C^*$ where the boundary is the Hopf link. It has Milnor number 1 and multiplicity 2. The singularity type of $f$ is $A_2$ and the Milnor fiber is a once-punctured torus with boundary being the trefoil. It's Milnor number is 2 and its multiplicity is 2. Figure 1 shows the embedding of $M_g$ into $M_f$.
		
		\begin{center}
			\tikzset{every picture/.style={line width=0.75pt}} %set default line width to 0.75pt        
			\begin{tikzpicture}[x=0.75pt,y=0.75pt,yscale=-1,xscale=1]
				
				%Curve Lines [id:da8330148229554024] 
				\draw    (234.33,67.5) .. controls (308.33,66.5) and (288.33,109.5) .. (332.33,109.5) ;
				%Curve Lines [id:da10476641667971154] 
				\draw    (140.33,144.5) .. controls (146.33,174.5) and (188.33,176.5) .. (234.33,176.5) ;
				%Curve Lines [id:da5120597559308806] 
				\draw    (140.33,98.5) .. controls (146.33,68.5) and (187.33,66.5) .. (234.33,67.5) ;
				%Curve Lines [id:da47838755731856164] 
				\draw    (140.33,144.5) .. controls (136.33,134.5) and (137.33,110.5) .. (140.33,98.5) ;
				%Curve Lines [id:da12550014911949714] 
				\draw    (188.33,125.5) .. controls (202.33,145.5) and (229.33,144.5) .. (244.33,124.5) ;
				%Curve Lines [id:da4105487135953434] 
				\draw    (194,128) .. controls (202.33,109.5) and (227.33,110.5) .. (237.33,127.5) ;
				%Shape: Ellipse [id:dp4019603959685496] 
				\draw   (325.83,127) .. controls (325.83,117.34) and (328.74,109.5) .. (332.33,109.5) .. controls (335.92,109.5) and (338.83,117.34) .. (338.83,127) .. controls (338.83,136.66) and (335.92,144.5) .. (332.33,144.5) .. controls (328.74,144.5) and (325.83,136.66) .. (325.83,127) -- cycle ;
				%Curve Lines [id:da7660791007758514] 
				\draw    (234.33,176.5) .. controls (308.33,176.5) and (289.33,143.5) .. (332.33,144.5) ;
				%Shape: Ellipse [id:dp1264202171445128] 
				\draw  [dash pattern={on 0.84pt off 2.51pt}] (206.33,90.5) .. controls (206.33,77.8) and (209.92,67.5) .. (214.33,67.5) .. controls (218.75,67.5) and (222.33,77.8) .. (222.33,90.5) .. controls (222.33,103.2) and (218.75,113.5) .. (214.33,113.5) .. controls (209.92,113.5) and (206.33,103.2) .. (206.33,90.5) -- cycle ;
				%Shape: Ellipse [id:dp8973638996047228] 
				\draw  [dash pattern={on 0.84pt off 2.51pt}] (207.33,158.5) .. controls (207.33,148.56) and (210.69,140.5) .. (214.83,140.5) .. controls (218.98,140.5) and (222.33,148.56) .. (222.33,158.5) .. controls (222.33,168.44) and (218.98,176.5) .. (214.83,176.5) .. controls (210.69,176.5) and (207.33,168.44) .. (207.33,158.5) -- cycle ;
				% Text Node
				\draw (156,113.4) node [anchor=north west][inner sep=0.75pt]    {$M_{g}$};
				% Text Node
				\draw (266,113.4) node [anchor=north west][inner sep=0.75pt]    {$M_{f}$};
			\end{tikzpicture}
			
			\caption{Figure 1: $M_g$ embeds into $M_f$}
		\end{center}
	\end{example}
	
	\begin{example}
		Let $f_t = x^2(x+t) + y^2(y^2+t)$. Observe that when $t=0$, we have $x^3+y^4$ which defines a $E_6$ singularity. When $t \neq 0$, then the singularity at $(0,0)$ is of $A_1$ type. We have that $\mu(f_0)=6,\nu(f_0)=3$ and the Milnor fiber of $f_0$ has Euler characteristic -5 with boundary being the (connected) torus knot $T(4,3)$ (also called $8_{19}$). Since $\chi = 2-2g-b$ where $b$ is the number of boundary components, we find that the genus is 3.
		
		We can perform four blowups to obtain a log resolution and in the process, look at the different chars to get the multiplicities. For example, in one of the charts at the end, the equation is $a^6 b^2(a^6 b + a^6 b^2 + a^2 t + t)=0$. When $t=0$, this simplifies to $a^{12}b^3(1+b)$. Below is a picture of the normal crossings divisor for the two singularities, labeled with their order of vanishing. The lowest order of vanishing is the multiplicity.
		
		\begin{center}
			\tikzset{every picture/.style={line width=0.75pt}} %set default line width to 0.75pt        
			
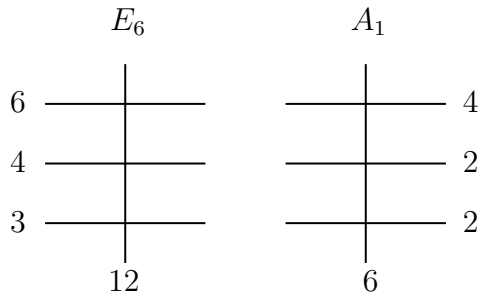
			\begin{tikzpicture}[x=0.75pt,y=0.75pt,yscale=-1,xscale=1]
				
				%Straight Lines [id:da6370846299441582] 
				\draw    (200,80) -- (200,180) ;
				%Straight Lines [id:da21607580686606154] 
				\draw    (160,100) -- (240,100) ;
				%Straight Lines [id:da9617709766732168] 
				\draw    (160,160) -- (240,160) ;
				%Straight Lines [id:da9788364879343552] 
				\draw    (160,130) -- (240,130) ;
				%Straight Lines [id:da5822673242014607] 
				\draw    (320,80) -- (320,180) ;
				%Straight Lines [id:da2848624596432199] 
				\draw    (280,100) -- (360,100) ;
				%Straight Lines [id:da2534125583407305] 
				\draw    (280,160) -- (360,160) ;
				%Straight Lines [id:da3873272440778157] 
				\draw    (280,130) -- (360,130) ;
				
				% Text Node
				\draw (141,92.4) node [anchor=north west][inner sep=0.75pt]    {$6$};
				% Text Node
				\draw (141,122.4) node [anchor=north west][inner sep=0.75pt]    {$4$};
				% Text Node
				\draw (141,152.4) node [anchor=north west][inner sep=0.75pt]    {$3$};
				% Text Node
				\draw (190,182.4) node [anchor=north west][inner sep=0.75pt]    {$12$};
				% Text Node
				\draw (367,92.4) node [anchor=north west][inner sep=0.75pt]    {$4$};
				% Text Node
				\draw (367,122.4) node [anchor=north west][inner sep=0.75pt]    {$2$};
				% Text Node
				\draw (367,152.4) node [anchor=north west][inner sep=0.75pt]    {$2$};
				% Text Node
				\draw (317,182.4) node [anchor=north west][inner sep=0.75pt]    {$6$};
				% Text Node
				\draw (191,50.4) node [anchor=north west][inner sep=0.75pt]    {$E_{6}$};
				% Text Node
				\draw (311,50.4) node [anchor=north west][inner sep=0.75pt]    {$A_{1}$};
			\end{tikzpicture}
		\end{center}				
	\end{example}
	
	\begin{remark}
		These two examples feature singularities of ADE type (also called simple singularities, Kleinian singularities, or du Val singularities) which have been studied extensively. They can be obtained by taking a finite subgroup $G \subset SL(2,\C)$ and quotienting to get $\C^2/G = \text{Spec}\, \C[x,y]^G$. Here, $\C[x,y]^G$ are the polynomials invariant under $G$ action. For instance, $A_1$ arises from $G= \Z/2 = \{\pm \id \}$ and we note that $x^2,xy,y^2$ are all invariant under this action. We also note that the $\Z/2$ action $(x,y) \mapsto (y,x)$ does not give a subgroup of $SL(2,\C)$ but of $GL(2,\C)$ and in this case, $\C^2/G \cong \C^2$ which shows that the analytic type does not remember information of the group $G$ when the action is in $GL(2,\C)$. 
	\end{remark}
	
	\subsection{Examples: ADE Singularities}
	
	The first purpose of this section is meant to provide the reader with some examples in order to give later sections some concreteness. It was proved by Arnold that among singularities of ADE type, two are adjacent if and only if the Dynkin diagram of one embeds into the Dynkin diagram of the other. In the case of complex surface singularities, this result can also be recovered sympletically and was probably known to Arnold. We outline a different proof strategy in order to fulfill a second purpose: highlight a few important results in the literature. Firstly, Castelnuovo proved the uniqueness of minimal resolutions of complex surface singularities; i.e. resolutions where all $(-1)$-curves have been contracted. We then use the results of Brieskorn \cite{brieskorn} and Ohta-Ono \cite{OhtaOno}:
	
	\begin{theorem} (Brieskorn)
		The minimal resolution of an isolated complex surface singularity and its Milnor fiber are diffeomorphic as a consequence of existence of simultaneous resolution.
	\end{theorem}
	
	\begin{theorem} (Ohta-Ono)
		Let $X$ be any minimal symplectic filling of the link of a simple surface singularity. Then, the diffeomorphism type of $X$ is unique. Hence, it must be diffeomorphic to the Milnor fiber. Moreover, the symplectic deformation type of $X$ is unique.
	\end{theorem}
	
	\noindent \textbf{Remark:} Recall that two symplectic forms $\omega_0,\omega_1$ are \textbf{symplectic deformation equivalent} if there exists a diffeomorphism $\phi$ such that $\phi^* \omega_1$ and $\omega_0$ are isotopic through a family of (not necessarily cohomologous) symplectic forms. In the proof of the theorem, Ohta-Ono attach a symplectic cap onto the fillings; these objects have many possible symplectic structures though they are all symplectic deformation equivalent.
	One can observe that the minimal resolution contains closed holomorphic curves and hence, cannot be a Stein filling. The Milnor filling, on the other hand, is Stein. One way to turn the minimal resolution into a Stein filling is to apply a hyperK\"ahler rotation which gives some kind of symplectic deformation and might not be the one that brings us to the Milnor filling, but at least it would turn the holomorphic spheres into Lagrangian spheres.
	In conversations with Kaoru Ono, he indicated that the result can be strengthened from one concerning symplectic deformation. The work of Lalonde-McDuff \cite{mcduff_lalonde} can be used to pinpoint the cohomology class of the symplectic structure.
	
	\begin{corollary}
		If two simple complex surface singularities are adjacent $[f_0]\to[f_1]$, then the Dynkin diagram for $f_1$ embeds into the Dynkin diagram for $f_0$.
	\end{corollary}
	
	\begin{proof}[Proof sketch]
		Suppose that $f_0,f_1:\C^3 \to \C$ are polynomials defining adjacent simple singularities; we will view $f_0$ as the "worse" singularity. The adjacency allows us to get a Stein cobordism between the link $L_0$ and $L_1$. Moreover, we can cap off the cobordism by using the unique minimal symplectic filling of $L_1$ to get a filling for $L_0$. Call these $X_1$ and $X_0$. This filling is also a minimal symplectic filling and hence, must be unique. By construction, the filling for $L_1$ embeds into the filling for $L_0$. Now, the Dynkin diagram for $L_1$ is given by studying the exceptional divisors and their intersections. This gives us a graph and the graph is precisely a Dynkin diagram of ADE type. Since we have an embedding $X_1 \hookrightarrow X_0$, we have an embedding $H_2(X_1,\Z) \hookrightarrow H_2(X_0,\Z)$ which respects the intersection form. One needs to do some linear algebra since its possible that the basis for $X_1$ is not sent to the basis for $X_0$ but rather some linear combination. But this is doable and once done, this implies that the Dynkin diagram for $X_1$ embeds into the Dynkin diagram for $X_0$
	\end{proof}
	
	\subsection{Embedding Milnor Fibers of Adjacent Singularities}
	
	Now, if two hypersurface singularities are adjacent, then there is a natural cobordism between their mapping tori and also a natural cobordism between their Milnor fibrations. We'll first construct a cobordism between their mapping tori. Suppose we have a 1-parameter family of polynomials $f_t$ giving us the adjacency of $f_0$ to $f_\eta$ for small $\eta$; such a family might not always have good algebraic properties such as being flat along $t$. However, we won't need such properties because our study is from a symplectic perspective. We will symplectically modify the zero loci via bump functions so that they agree near the boundary of $B_\epsilon$ in the following way. Let $\beta:\C^{n+1} \to \R$ be a bump function supported in the region where the radius is $\epsilon' \leq r \leq \epsilon$; we can assume that $\beta = 1$ for $r > (\epsilon'+\epsilon)/2$ and is radially symmetric. We can later specify what $\epsilon'$ is but it will be very near $\epsilon$ and should be chosen such that the compact support of the monodromy of $g$ does not intersect $S^{2n+1}_{\epsilon'}$. Then, letting $g = f_\eta$ for $\eta$ extremely small, we define $\tilde{g} = g+(f-g)\beta$.
	
	\begin{lemma} \label{embed}
		Let $f,g$ be adjacent as above. Then the Milnor fiber of $g$ smoothly embeds into the Milnor fiber of $\tilde{g}$ which smoothly embeds into the Milnor fiber of $f$.
	\end{lemma}
	
	\begin{proof}
		We take this opportunity to recall some basic facts. The Milnor fiber of a polynomial with isolated singularity is isomorphic to the zero locus of a generic smoothing of $f$. That is, if $\mu$ is the Milnor number which is the dimension of $\C[z_1,...,z_{n+1}]/\text{Jac}(f)$ where $\text{Jac}(f)$ is the Jacobian ideal of $f$, then we have a miniversal deformation space $\C^\mu$. Take a small ball $B \subset \C^\mu$ centered at 0. A tuple $(\eta_1,...,\eta_\mu) \in B$ gives us $\eta = \sum \eta_i p_i$ where $p_i$ is a basis for $\C[z_1,...,z_{n+1}]/\text{Jac}(f)$. We construct a fibration over $B$ where the fiber is the zero locus of $f+\eta$ intersected with a small ball in $\C^{n+1}$. Then, when $\eta = 0$, this is our singularity and for generic $\eta$ (one not in the discriminant locus which has codim 1), the fiber is isomorphic to the Milnor fiber. So in this sense, a generic $\eta$ gives a generic smoothing of $f$.
		
		Because we have an adjacency, what we can do is first perturb $f$ so that we get the singularity type of $g$ and then perturb again to get a smoothing of $g$. Since the latter step is done by choosing a generic $\eta$, we see that this is also a generic smoothing of $f$. The way to get the embedding is to use different radii of balls as our cutoff. Let $\hat{\epsilon} \ll \epsilon$. For the Milnor fiber of $g$, it is the zero locus of a smoothing of $g$ intersected with $B_{\hat{\epsilon}}$. This embeds into the zero locus of the same smoothing of $g$ intersected with a larger ball $B_\epsilon$. The latter is also a smoothing of $f$. Hence, the Milnor fiber of $g$ embeds into the Milnor fiber of $f$. For a similar proof, see \cite{keating}.
		
		Observe also that in the modification above to obtain $\tilde{g}$, since it happens very near the boundary of $B_\epsilon$, these smoothings can be chosen to not affect $f$ and $\tilde{g}$ near $\partial B_\epsilon$. For example, $f^{-1}(0), \tilde{g}^{-1}(0)$ are transverse to $\partial B_\epsilon$. As such, the Milnor fiber $M_g$ also embeds into the Milnor fiber $M_{\tilde{g}}$ which embeds into the Milnor fiber $M_f$. Moreover, $M_g$ and $M_{\tilde{g}}$ are diffeomorphic as one can construct a vector field whose flow maps $M_{\tilde{g}}$ onto $M_g$. Here is a picture.
		
		\begin{center}
			\tikzset{every picture/.style={line width=0.75pt}} %set default line width to 0.75pt        
			
			\begin{tikzpicture}[x=0.75pt,y=0.75pt,yscale=-1,xscale=1]
				%uncomment if require: \path (0,372); %set diagram left start at 0, and has height of 372
				
				%Curve Lines [id:da5420547952716865] 
				\draw [line width=1.5]    (355.42,224.86) .. controls (367.33,257.5) and (414.77,206.09) .. (436.57,262.62) ;
				%Curve Lines [id:da24025337540987635] 
				\draw [line width=1.5]    (355.54,111.81) .. controls (363.33,86.5) and (375.81,89.52) .. (385.38,89.76) .. controls (394.95,90.01) and (423.33,113.17) .. (433.77,64.84) ;
				%Curve Lines [id:da6739804758408963] 
				\draw    (349.22,211.33) .. controls (371.4,256.44) and (381.54,279.12) .. (391.33,301.5) ;
				%Curve Lines [id:da21397622713626885] 
				\draw    (316.83,168.1) .. controls (388.19,164.14) and (422.57,211.62) .. (436.57,262.62) ;
				%Curve Lines [id:da8731683817333415] 
				\draw    (316.83,168.1) .. controls (396.58,147.36) and (413.33,116.5) .. (432.77,70.4) ;
				%Shape: Ellipse [id:dp1581194322078936] 
				\draw  [dash pattern={on 0.84pt off 2.51pt}] (164.33,168.1) .. controls (164.33,83.91) and (232.61,15.66) .. (316.83,15.66) .. controls (401.06,15.66) and (469.33,83.91) .. (469.33,168.1) .. controls (469.33,252.29) and (401.06,320.54) .. (316.83,320.54) .. controls (232.61,320.54) and (164.33,252.29) .. (164.33,168.1) -- cycle ;
				%Shape: Ellipse [id:dp6585130852464707] 
				\draw  [dash pattern={on 0.84pt off 2.51pt}] (248.63,168.1) .. controls (248.63,130.45) and (279.16,99.92) .. (316.83,99.92) .. controls (354.5,99.92) and (385.04,130.45) .. (385.04,168.1) .. controls (385.04,205.75) and (354.5,236.28) .. (316.83,236.28) .. controls (279.16,236.28) and (248.63,205.75) .. (248.63,168.1) -- cycle ;
				%Shape: Ellipse [id:dp5335788850778409] 
				\draw  [color={rgb, 255:red, 255; green, 0; blue, 0 }  ,draw opacity=1 ][fill={rgb, 255:red, 255; green, 0; blue, 0 }  ,fill opacity=1 ] (350.64,224.86) .. controls (350.64,222.22) and (352.78,220.08) .. (355.42,220.08) .. controls (358.06,220.08) and (360.2,222.22) .. (360.2,224.86) .. controls (360.2,227.5) and (358.06,229.64) .. (355.42,229.64) .. controls (352.78,229.64) and (350.64,227.5) .. (350.64,224.86) -- cycle ;
				%Shape: Ellipse [id:dp13121358083732826] 
				\draw  [color={rgb, 255:red, 0; green, 0; blue, 255 }  ,draw opacity=1 ][fill={rgb, 255:red, 0; green, 0; blue, 255 }  ,fill opacity=1 ] (428.99,69.62) .. controls (428.99,66.98) and (431.13,64.84) .. (433.77,64.84) .. controls (436.41,64.84) and (438.55,66.98) .. (438.55,69.62) .. controls (438.55,72.26) and (436.41,74.4) .. (433.77,74.4) .. controls (431.13,74.4) and (428.99,72.26) .. (428.99,69.62) -- cycle ;
				%Curve Lines [id:da5880687424034943] 
				\draw    (232.33,135.5) .. controls (255.33,106.5) and (324.78,159.72) .. (349.22,211.33) ;
				%Curve Lines [id:da043699630475921625] 
				\draw    (234.33,202.5) .. controls (263.33,228.5) and (326.48,171.98) .. (351.35,119.39) ;
				%Curve Lines [id:da9402739527339292] 
				\draw    (351.35,119.39) .. controls (365.8,87.22) and (372.94,78.25) .. (388.33,33.5) ;
				%Shape: Ellipse [id:dp4548857614142561] 
				\draw  [color={rgb, 255:red, 255; green, 0; blue, 0 }  ,draw opacity=1 ][fill={rgb, 255:red, 255; green, 0; blue, 0 }  ,fill opacity=1 ] (350.76,111.81) .. controls (350.76,109.17) and (352.9,107.03) .. (355.54,107.03) .. controls (358.18,107.03) and (360.32,109.17) .. (360.32,111.81) .. controls (360.32,114.45) and (358.18,116.59) .. (355.54,116.59) .. controls (352.9,116.59) and (350.76,114.45) .. (350.76,111.81) -- cycle ;
				%Shape: Ellipse [id:dp15238140945376166] 
				\draw  [color={rgb, 255:red, 0; green, 0; blue, 255 }  ,draw opacity=1 ][fill={rgb, 255:red, 0; green, 0; blue, 255 }  ,fill opacity=1 ] (431.79,262.62) .. controls (431.79,259.98) and (433.93,257.84) .. (436.57,257.84) .. controls (439.21,257.84) and (441.35,259.98) .. (441.35,262.62) .. controls (441.35,265.26) and (439.21,267.4) .. (436.57,267.4) .. controls (433.93,267.4) and (431.79,265.26) .. (431.79,262.62) -- cycle ;
				%Curve Lines [id:da29975410013770265] 
				\draw    (234.33,202.5) .. controls (216.33,189.5) and (218.33,155.5) .. (232.33,135.5) ;
				%Shape: Circle [id:dp011268652518372502] 
				\draw  [fill={rgb, 255:red, 0; green, 0; blue, 0 }  ,fill opacity=1 ] (315.58,167.85) .. controls (315.58,166.61) and (316.59,165.6) .. (317.83,165.6) .. controls (319.08,165.6) and (320.08,166.61) .. (320.08,167.85) .. controls (320.08,169.09) and (319.08,170.1) .. (317.83,170.1) .. controls (316.59,170.1) and (315.58,169.09) .. (315.58,167.85) -- cycle ;
				
				% Text Node
				\draw (416,187.4) node [anchor=north west][inner sep=0.75pt]    {$f$};
				% Text Node
				\draw (207,161.4) node [anchor=north west][inner sep=0.75pt]    {$g$};
				% Text Node
				\draw (395,66.4) node [anchor=north west][inner sep=0.75pt]    {$\tilde{g}$};
			\end{tikzpicture}
			
			\caption{Figure 2: A modification}
		\end{center} 
		
	\end{proof}
	
	Next, we prove a symplectic result about the Milnor fiber of $\tilde{g}$.
	
	\begin{lemma}
		Under the modification, $M_{\tilde{g}}$ still carries a symplectic structure and the symplectic type is independent of the choice of bump function $\beta$ from Lemma \ref{embed}. Moreover, $M_{\tilde{g}}$ and $M_g$ are exact symplectomorphic. 
	\end{lemma}
	
	\begin{proof}
		By virtue of being smoothly embedded in $\C^{n+1}$, $g^{-1}(0)$ and $\tilde{g}^{-1}(0)$ both have exact symplectic forms (away from the singularity) and the modification we've described. Observe that though $\beta$ is not $C^1$-small, $f-g$ is $C^1$-small, as small as we would like by tuning the $\eta$ from the proof of Lemma \ref{embed}. Hence, $(f-g)\beta$ is $C^1$-small which means the polynomials $\tilde{g}$ and $g$ are $C^1$-close.
		
		If $\beta_1,\beta_2$ are two bump functions that are radially symmetric, then we may as well think of them as functions $\beta_1,\beta_2: [\epsilon',\epsilon] \to [0,1]$ with non-negative first derivative. Then, if $t \in [0,1]$, $(1-t) \beta_1 + t \beta_2$ is also a bump function in that it equals 0 and 1 in the appropriate regimes and the first derivative is non-negative. This isotopy in $t$ gives us a family of symplectic fibers. Since the symplectic forms are all exact, the Moser lemma shows that the fibers are all symplectomorphic (in the proof of Lemma \ref{swatch}, this is outlined). Lastly, we cite a result found in Cieliebak-Eliashberg's book; the proof is not difficult.
		
		\begin{lemma}\label{gavela} (\cite{Cieliebak-Eliashberg},Lemma 11.2)
			Any symplectomorphism $f:(V,\lambda) \to (V',\lambda')$ between finite type Liouville manifolds is diffeomorphic to an exact symplectomorphism.
		\end{lemma}
		
		To apply this lemma, we point out that Milnor fibers are finite type Stein domains and we may complete them in the usual way to Stein (and hence, Liouville) manifolds.		
	\end{proof}
	
	\begin{lemma} \label{swatch}
	The fixed-point Floer cohomology $HF^*$ is invariant under this modification.
	\end{lemma}
	
	\begin{proof}
		By tuning $\eta$ to zero, what we actually get is a smooth family of mapping tori. We want $HF^*$ to be invariant under this smooth family. At this point, it is more convenient to work with the classical Milnor fibration where the corresponding objects form a smooth family of Liouville domains where the variation takes place near the contact boundary. We may choose a smooth path in the disk and hence, obtain a smooth 1-parameter family of such Liouville domains between the Milnor fibers of $f$ and $g$. By Gray's stability (see, for example, McDuff-Salamon \cite{McDuffSalamon}, p. 136), when we have a 1-parameter family (in $\tau$) of contact forms $\alpha_\tau$ on a closed manifold, in this case, the link $L$ considered as a smooth submanifold, there exists an isotopy $\psi_\tau$ and a family of functions $f_\tau$ such that $\alpha_\tau = \psi^*_\tau(f_\tau \alpha_0)$ where $\psi_0 = \id$ and $f_0 \equiv 1$. Hence, the $\psi_\tau$ are contactomorphisms since $d\psi_\tau$ maps $\ker \alpha_\tau$ to $\ker \alpha_0$; the scaling by $f_\tau$ doesn't affect the contact structure.
		
		Similarly, we have a family of symplectic forms $\omega_\tau = d\alpha_\tau$ with exact derivative. Hence, we can use the Moser lemma to show that there is an isotopy $\Psi_\tau$ such that $\omega_\tau = \Psi^*_\tau \omega_0$ with $\Psi_0 = \id$ and $\Psi_\tau|_L = \psi_\tau$. Lastly, $HF^*$ is a symplectic invariant and hence, for all $\tau$, the corresponding $HF^*$'s of the $\phi_\tau:(F,\omega_\tau) \to (F,\omega_\tau)$ are all isomorphic. cf. property 2 in McLean's paper.
	\end{proof}
	
	Now, the links are contact submanifolds of the respective $(2n+1)$-spheres. Whenever we have a contact embedding $j:L \hookrightarrow (M,\alpha)$ of a submanifold into a contact manifold $(M,\alpha)$ where $\alpha$ is the contact form, some basic linear algebra shows us that the contact structure $\xi_L$ embeds into $\xi_M$ as a symplectic subbundle. Moreover, the Reeb field for $L$ is the same as the Reeb field for $M$. Hence, $\xi_M = \xi_L \oplus \xi_L^\perp$ and $\xi_L^\perp$ is the normal bundle of $L$ inside of $S^{2n+1}$. Let's now return to our situation where $L$ is the link of an isolated singularity and $M = S^{2n+1}$.
	
	%	\begin{lemma}
		%		$\xi_M \to L$ is a complex vector bundle and admits a complex section.
		%	\end{lemma}
	
	%	\begin{proof}
		%		$\xi_M \to L$ is a rank $2n$ symplectic vector bundle and hence, a complex bundle. It is classified by a map $f:L \to BU(n)$. The existence of a complex section is assured if we can lift $f$ to a map $L \to BU(n-1)$. Basic obstruction theory says that the obstructions live in $H^*(L,\pi_{*-1}S^{2n-1})$. This is because the homotopy fiber is $U(n)/U(n-1) \cong S^{2n-1}$. Hence, the first obstruction is when $* = 2n$. But $H^{2n}(L,\Z) = 0$.
		%	\end{proof}
	
	\begin{lemma} \label{normal}
		Let $f:\C^{n+1} \to \C$ be a holomorphic function with isolated singularity at 0. The complex normal bundle of $F=f^{-1}(0) \setminus 0$ is trivialized by $df$ and hence, the normal bundle of $L_f:=f^{-1}(0) \cap S^{2n+1}_\epsilon$ also has a canonical trivialization induced by $f$.
	\end{lemma}
	
	\begin{proof}
		Since $f:\C^{n+1} \to \C$ is a holomorphic map and 0 is the only point where $df=0$, then $df:T\C^{n+1} \to T\C$ is surjective everywhere else. Since $F:=f^{-1}(0) \setminus 0$ is a symplectic submanifold and part of the fiber, $df$ maps $TF$ to zero and the normal bundle of $F$ maps surjectively onto $T\C$. In particular, since $L_f = F \cap S^{2n+1}_\epsilon$ is a transversal intersection of $F$ with a regular level set of the norm squared function $|\cdot|^2:\C^{n+1} \to \R$, the normal bundle of $L_f$ in $S^{2n+1}_\epsilon$ is also trivialized.
	\end{proof}
	
	Let $W$ be the portion of $\tilde{g}^{-1}(0)$ whose points have radius between $\epsilon'$ and $\epsilon$. This means that the boundary of $W$ is the union of the links $L_g$ and $L_f$. In \cite{Milnor}, Milnor showed that Milnor fibers are smoothly parallelizable and that is the case here. In fact, a much stronger statement is true: the tangent bundle is algebraically trivial; we'll show why below. Moreover, for a generic fixed radius $r$, the level set $W_r = W \cap S^{2n+1}_r$ is a link with trivial normal bundle in $S^{2n+1}_r$. This shows that in fact, $W$ itself has trivial normal bundle in $B^{2n+2}_\epsilon$ due to the adjacency of $f$ and $g$. We call this the \textbf{canonical normal framing}.
	
	Now, to prove that the tangent bundle of the Milnor fiber is algebraically trivial and hence, also trivial as both a complex and holomorphic bundle, we may use a powerful theorem of Suslin \cite{suslin}. This is much stronger than we need but we wish to advertise this result which may be less familiar to symplectic geometers.
	
	\begin{theorem}
		Every stably trivial vector bundle of rank $n$ on a smooth affine variety of dimension $n$ over an algebraically closed field is algebraically trivial.
	\end{theorem}
	
	The proof of Lemma \ref{normal} shows that any smooth complex affine hypersurface $F \subset \C^{n+1}$ has a trivial normal bundle $\nu$ (and indeed, the same holds for complete intersections) and hence, has stably trivial tangent bundle since $TF \oplus \nu = \C^{n+1}$. Suslin's theorem then shows that $TF$ is algebraically trivial.
	
	\subsection{A Cobordism for the Milnor Fibration}
	
	Since we showed above that $W$ has trivial normal bundle, then the circle bundle of the normal bundle is $W \times S^1$. The cobordism that we construct will simply be the annulus $B^{2n+2}_\epsilon \setminus B^{2n+2}_{\epsilon'}$ with the normal bundle of $W$ removed. This is a cobordism with corners between the Milnor fibrations; its boundary naturally decomposes into a horizontal and vertical part. $W \times S^1$ is the horizontal part and it is a contact hypersurface because it inherits its contact structure from the mapping torus. This structure is not $\xi = TW$ since that is integrable and stable Hamiltonian but not contact. Figure 2 depicts this; the red circles represent $L_g$ and the blue circles represent $L_f$. The dotted lines denote the boundary and the bolded lines represent $W$. \textit{Caveat lector}: Figure 2 portrays the real parts of a nodal and cuspidal singularity but in general, $L_g,L_f, W$ are all connected. Also, the picture is drawn so that the zero loci appear to be far apart as we move further from the origin but that's for sake of having a picture that isn't too crowded.
	
	\begin{center}
		\tikzset{every picture/.style={line width=0.75pt}} %set default line width to 0.75pt
		
		\begin{tikzpicture}[x=0.75pt,y=0.75pt,yscale=-1,xscale=1]
			%uncomment if require: \path (0,372); %set diagram left start at 0, and has height of 372
			
			%Curve Lines [id:da5420547952716865]
			\draw [line width=1.5]    (355.42,224.86) .. controls (367.33,257.5) and (414.77,206.09) .. (436.57,262.62) ;
			%Curve Lines [id:da24025337540987635]
			\draw [line width=1.5]    (355.54,111.81) .. controls (363.33,86.5) and (375.81,89.52) .. (385.38,89.76) .. controls (394.95,90.01) and (423.33,113.17) .. (433.77,64.84) ;
			%Curve Lines [id:da6739804758408963]
			\draw    (349.22,211.33) .. controls (371.4,256.44) and (381.54,279.12) .. (391.33,301.5) ;
			%Curve Lines [id:da21397622713626885]
			\draw    (316.83,168.1) .. controls (388.19,164.14) and (422.57,211.62) .. (436.57,262.62) ;
			%Curve Lines [id:da8731683817333415]
			\draw    (316.83,168.1) .. controls (396.58,147.36) and (413.33,116.5) .. (432.77,70.4) ;
			%Shape: Ellipse [id:dp1581194322078936]
			\draw  [dash pattern={on 0.84pt off 2.51pt}] (164.33,168.1) .. controls (164.33,83.91) and (232.61,15.66) .. (316.83,15.66) .. controls (401.06,15.66) and (469.33,83.91) .. (469.33,168.1) .. controls (469.33,252.29) and (401.06,320.54) .. (316.83,320.54) .. controls (232.61,320.54) and (164.33,252.29) .. (164.33,168.1) -- cycle ;
			%Shape: Ellipse [id:dp6585130852464707]
			\draw  [dash pattern={on 0.84pt off 2.51pt}] (248.63,168.1) .. controls (248.63,130.45) and (279.16,99.92) .. (316.83,99.92) .. controls (354.5,99.92) and (385.04,130.45) .. (385.04,168.1) .. controls (385.04,205.75) and (354.5,236.28) .. (316.83,236.28) .. controls (279.16,236.28) and (248.63,205.75) .. (248.63,168.1) -- cycle ;
			%Shape: Ellipse [id:dp5335788850778409]
			\draw  [color={rgb, 255:red, 255; green, 0; blue, 0 }  ,draw opacity=1 ][fill={rgb, 255:red, 255; green, 0; blue, 0 }  ,fill opacity=1 ] (350.64,224.86) .. controls (350.64,222.22) and (352.78,220.08) .. (355.42,220.08) .. controls (358.06,220.08) and (360.2,222.22) .. (360.2,224.86) .. controls (360.2,227.5) and (358.06,229.64) .. (355.42,229.64) .. controls (352.78,229.64) and (350.64,227.5) .. (350.64,224.86) -- cycle ;
			%Shape: Ellipse [id:dp13121358083732826]
			\draw  [color={rgb, 255:red, 0; green, 0; blue, 255 }  ,draw opacity=1 ][fill={rgb, 255:red, 0; green, 0; blue, 255 }  ,fill opacity=1 ] (428.99,69.62) .. controls (428.99,66.98) and (431.13,64.84) .. (433.77,64.84) .. controls (436.41,64.84) and (438.55,66.98) .. (438.55,69.62) .. controls (438.55,72.26) and (436.41,74.4) .. (433.77,74.4) .. controls (431.13,74.4) and (428.99,72.26) .. (428.99,69.62) -- cycle ;
			%Curve Lines [id:da5880687424034943]
			\draw    (232.33,135.5) .. controls (255.33,106.5) and (324.78,159.72) .. (349.22,211.33) ;
			%Curve Lines [id:da043699630475921625]
			\draw    (234.33,202.5) .. controls (263.33,228.5) and (326.48,171.98) .. (351.35,119.39) ;
			%Curve Lines [id:da9402739527339292]
			\draw    (351.35,119.39) .. controls (365.8,87.22) and (372.94,78.25) .. (388.33,33.5) ;
			%Shape: Ellipse [id:dp4548857614142561]
			\draw  [color={rgb, 255:red, 255; green, 0; blue, 0 }  ,draw opacity=1 ][fill={rgb, 255:red, 255; green, 0; blue, 0 }  ,fill opacity=1 ] (350.76,111.81) .. controls (350.76,109.17) and (352.9,107.03) .. (355.54,107.03) .. controls (358.18,107.03) and (360.32,109.17) .. (360.32,111.81) .. controls (360.32,114.45) and (358.18,116.59) .. (355.54,116.59) .. controls (352.9,116.59) and (350.76,114.45) .. (350.76,111.81) -- cycle ;
			%Shape: Ellipse [id:dp15238140945376166]
			\draw  [color={rgb, 255:red, 0; green, 0; blue, 255 }  ,draw opacity=1 ][fill={rgb, 255:red, 0; green, 0; blue, 255 }  ,fill opacity=1 ] (431.79,262.62) .. controls (431.79,259.98) and (433.93,257.84) .. (436.57,257.84) .. controls (439.21,257.84) and (441.35,259.98) .. (441.35,262.62) .. controls (441.35,265.26) and (439.21,267.4) .. (436.57,267.4) .. controls (433.93,267.4) and (431.79,265.26) .. (431.79,262.62) -- cycle ;
			%Curve Lines [id:da02062722520202498]
			\draw [line width=0.75]  [dash pattern={on 0.84pt off 2.51pt}]  (364.33,118.17) .. controls (374.58,103.53) and (381.81,102.52) .. (391.38,102.76) .. controls (400.95,103.01) and (430.33,117.17) .. (439.77,77.84) ;
			%Curve Lines [id:da26112729605040985]
			\draw [line width=0.75]  [dash pattern={on 0.84pt off 2.51pt}]  (345.54,105.03) .. controls (364.33,64.5) and (377.76,76.26) .. (387.33,76.5) .. controls (396.91,76.74) and (409.89,95.49) .. (424.33,61.17) ;
			%Curve Lines [id:da5499778128623531]
			\draw [line width=0.75]  [dash pattern={on 0.84pt off 2.51pt}]  (362.42,218.86) .. controls (379.33,243.5) and (420.77,199.09) .. (442.57,255.62) ;
			%Curve Lines [id:da7305752068209586]
			\draw [line width=0.75]  [dash pattern={on 0.84pt off 2.51pt}]  (345.42,231.86) .. controls (368.33,269.5) and (420.33,223.17) .. (429.33,271.17) ;
			%Curve Lines [id:da29975410013770265]
			\draw    (234.33,202.5) .. controls (216.33,189.5) and (218.33,155.5) .. (232.33,135.5) ;
			%Shape: Circle [id:dp011268652518372502]
			\draw  [fill={rgb, 255:red, 0; green, 0; blue, 0 }  ,fill opacity=1 ] (315.58,167.85) .. controls (315.58,166.61) and (316.59,165.6) .. (317.83,165.6) .. controls (319.08,165.6) and (320.08,166.61) .. (320.08,167.85) .. controls (320.08,169.09) and (319.08,170.1) .. (317.83,170.1) .. controls (316.59,170.1) and (315.58,169.09) .. (315.58,167.85) -- cycle ;
		\end{tikzpicture}
		
		\caption{Figure 2: A cobordism between Milnor fibrations}
	\end{center}
	
	\subsection{A Cobordism for the Mapping Tori}\label{cobordism}
	
	However, a more convenient cobordism is to use the mapping torus model of the Milnor fibration. Recall that we had a small $\eta$ from above where $g=f_\eta$ and we modified $g$ to obtain $\tilde{g} = g+(f-g)\beta$. We can construct a cobordism using the mapping torus model: $E=(B_\epsilon \cap f^{-1}(\bar{D}_{\delta_1})) \setminus \tilde{g}^{-1}(\mathring{D}_{\delta_0})$. Here, $0 < \delta_0 \ll \delta_1 \ll \epsilon$ and both $f$ and $\tilde{g}$ are submersions on $E$. Observe that this is also a symplectic manifold with corners; the symplectic form is inherited from $\C^{n+1}$ and the boundary decomposes into horizontal and vertical part.
	
	The vertical boundary is the union of the mapping tori of $f$ and $\tilde{g}$. We also have a symplectic collar neighborhood of the horizontal boundary because of how we constructed $\tilde{g}$.
	
	\begin{center}
		\tikzset{every picture/.style={line width=0.75pt}} %set default line width to 0.75pt        
		\begin{tikzpicture}[x=0.75pt,y=0.75pt,yscale=-1,xscale=1]
			%uncomment if require: \path (0,349); %set diagram left start at 0, and has height of 349
			
			%Shape: Ellipse [id:dp24374821846807682] 
			\draw  [dash pattern={on 0.84pt off 2.51pt}] (221,36.75) .. controls (221,21.42) and (267.41,9) .. (324.67,9) .. controls (381.92,9) and (428.33,21.42) .. (428.33,36.75) .. controls (428.33,52.08) and (381.92,64.5) .. (324.67,64.5) .. controls (267.41,64.5) and (221,52.08) .. (221,36.75) -- cycle ;
			%Curve Lines [id:da06790648651918696] 
			\draw  [dash pattern={on 0.84pt off 2.51pt}]  (439.67,204.5) .. controls (453.33,165.5) and (409.33,80.5) .. (428.33,36.75) ;
			%Curve Lines [id:da20816205512679198] 
			\draw  [dash pattern={on 0.84pt off 2.51pt}]  (232.33,204.5) .. controls (251.33,160.5) and (202,80.5) .. (221,36.75) ;
			%Shape: Ellipse [id:dp0444441153884918] 
			\draw  [dash pattern={on 0.84pt off 2.51pt}] (232.33,204.5) .. controls (232.33,189.17) and (278.75,176.75) .. (336,176.75) .. controls (393.25,176.75) and (439.67,189.17) .. (439.67,204.5) .. controls (439.67,219.83) and (393.25,232.25) .. (336,232.25) .. controls (278.75,232.25) and (232.33,219.83) .. (232.33,204.5) -- cycle ;
			%Shape: Ellipse [id:dp9526452298757178] 
			\draw  [dash pattern={on 0.84pt off 2.51pt}] (268.58,36.75) .. controls (268.58,29.78) and (293.69,24.13) .. (324.67,24.13) .. controls (355.64,24.13) and (380.75,29.78) .. (380.75,36.75) .. controls (380.75,43.72) and (355.64,49.38) .. (324.67,49.38) .. controls (293.69,49.38) and (268.58,43.72) .. (268.58,36.75) -- cycle ;
			%Shape: Ellipse [id:dp16383012150859466] 
			\draw  [dash pattern={on 0.84pt off 2.51pt}] (279.92,204.5) .. controls (279.92,197.53) and (305.03,191.88) .. (336,191.88) .. controls (366.97,191.88) and (392.08,197.53) .. (392.08,204.5) .. controls (392.08,211.47) and (366.97,217.13) .. (336,217.13) .. controls (305.03,217.13) and (279.92,211.47) .. (279.92,204.5) -- cycle ;
			%Curve Lines [id:da13939771510494858] 
			\draw  [dash pattern={on 0.84pt off 2.51pt}]  (268.58,36.75) .. controls (313.33,90.5) and (256.33,155.5) .. (279.92,204.5) ;
			%Curve Lines [id:da1991715281581663] 
			\draw  [dash pattern={on 0.84pt off 2.51pt}]  (380.75,36.75) .. controls (413.33,154.5) and (368.5,155.5) .. (392.08,204.5) ;
			%Shape: Ellipse [id:dp4931682884429913] 
			\draw   (232.33,308) .. controls (232.33,292.67) and (278.75,280.25) .. (336,280.25) .. controls (393.25,280.25) and (439.67,292.67) .. (439.67,308) .. controls (439.67,323.33) and (393.25,335.75) .. (336,335.75) .. controls (278.75,335.75) and (232.33,323.33) .. (232.33,308) -- cycle ;
			%Shape: Ellipse [id:dp5432773217987334] 
			\draw   (279.92,308) .. controls (279.92,301.03) and (305.03,295.38) .. (336,295.38) .. controls (366.97,295.38) and (392.08,301.03) .. (392.08,308) .. controls (392.08,314.97) and (366.97,320.63) .. (336,320.63) .. controls (305.03,320.63) and (279.92,314.97) .. (279.92,308) -- cycle ;
			%Straight Lines [id:da47010186120847886] 
			\draw    (336,239.25) -- (336,272.25) ;
			\draw [shift={(336,274.25)}, rotate = 270] [color={rgb, 255:red, 0; green, 0; blue, 0 }  ][line width=0.75]    (10.93,-3.29) .. controls (6.95,-1.4) and (3.31,-0.3) .. (0,0) .. controls (3.31,0.3) and (6.95,1.4) .. (10.93,3.29)   ;
			%Shape: Ellipse [id:dp10948276888675768] 
			\draw   (284.33,107.75) .. controls (284.33,100.78) and (308.51,95.13) .. (338.33,95.13) .. controls (368.16,95.13) and (392.33,100.78) .. (392.33,107.75) .. controls (392.33,114.72) and (368.16,120.38) .. (338.33,120.38) .. controls (308.51,120.38) and (284.33,114.72) .. (284.33,107.75) -- cycle ;
			%Curve Lines [id:da3574546404542658] 
			\draw [line width=1.5]    (268.58,36.75) .. controls (283.33,55.83) and (290.33,83.83) .. (284.33,107.75) ;
			%Shape: Circle [id:dp3071105744152871] 
			\draw  [color={rgb, 255:red, 0; green, 0; blue, 255 }  ,draw opacity=1 ][fill={rgb, 255:red, 0; green, 0; blue, 255 }  ,fill opacity=1 ] (266.58,36.75) .. controls (266.58,34.49) and (268.41,32.67) .. (270.67,32.67) .. controls (272.92,32.67) and (274.75,34.49) .. (274.75,36.75) .. controls (274.75,39.01) and (272.92,40.83) .. (270.67,40.83) .. controls (268.41,40.83) and (266.58,39.01) .. (266.58,36.75) -- cycle ;
			%Shape: Circle [id:dp1560319709664637] 
			\draw  [color={rgb, 255:red, 255; green, 0; blue, 0 }  ,draw opacity=1 ][fill={rgb, 255:red, 255; green, 0; blue, 0 }  ,fill opacity=1 ] (280.25,107.75) .. controls (280.25,105.49) and (282.08,103.67) .. (284.33,103.67) .. controls (286.59,103.67) and (288.42,105.49) .. (288.42,107.75) .. controls (288.42,110.01) and (286.59,111.83) .. (284.33,111.83) .. controls (282.08,111.83) and (280.25,110.01) .. (280.25,107.75) -- cycle ;
			
			% Text Node
			\draw (343,238.4) node [anchor=north west][inner sep=0.75pt]    {$\tilde{g}$};
		\end{tikzpicture}
		
		\caption{Figure 3: A cobordism between mapping tori}
	\end{center}
	
	\begin{remark}
		If we draw a tube around the nodal cubic $g^{-1}(0)$, then the picture suggests that it should form a real hypersurface that intersects itself. This immersed hypersurface, however, is \emph{not} the mapping torus of the monodromy map on the corresponding Milnor fibration.
		
		Another concern one may have is that the cuspidal cubic $f^{-1}(0)$ intersects the mapping torus associated to $g$ and hence, the two mapping tori intersect. However, because they are adjacent, $f,g$ are $C^1$-close and so, for small constants $\delta_0,\delta_1,\epsilon$ as above, the mapping tori do not intersect and $E$ is an honest manifold-with-corners. For concreteness, we used these examples but this discussion is true for any pair of adjacent singularities.
	\end{remark}
	
	Now, the situation above was such that we constructed a cobordism for the monodromy map $\phi$. However, if we want to construct cobordisms for higher iterates of $\phi$, we cannot use $\C^{n+1}$ but rather some branched covers. Again, consider $f:\C^{n+1} \to \C$ such that $f(0)=0$ is the only singularity of $f$ contained in $f^{-1}(0)$ and let $\phi$ be the associated monodromy obtained by symplectic parallel transport. We can construct more mapping tori by just taking iterates of the monodromy map. These tori no longer embed into $\mathbb{C}^n$ but do embed into $\C^{n+1}$ as follows. Consider $\C^n \setminus V(f)$, the complement of the zero locus of $f$. Then, it is isomorphic to the affine variety $V(wf-1) \subset \C^{n+1}$ where we've added a variable $w$ and so the coordinates are $(w,z_1,...,z_n)$ and the map $(w,\vec{z}) \mapsto \vec{z}$ is the isomorphism since, we see that $wf(\vec{z})=1$ implies that neither $w$ nor $f$ vanish on this set and $w$ is completely determined by $f$. This variety is $k$-fold covered by $V(w^kf-1) \subset \C^{n+1}$ via the map $(w,\vec{z}) \mapsto (w^k,\vec{z})$ without any branching. Since the mapping torus for the monodromy can now be viewed as embedded in $V(wf-1)$ via the isomorphism above, we can then take the preimage of this mapping torus to get something living in $V(w^kf-1)$ which is abstractly, the mapping torus for the $k$th iterate of the monodromy.
	
	So when we speak of a \textbf{branched mapping torus model}, we mean this particular model living in a branched covering of $\C^n$. Later on, we wish to consider a cobordism between two mapping tori; one knows that these exist since we can embed the mapping tori in some variety in $\C^{n+1}$ and then take the ``space between them.''
	
	\section{An Induced Map on $HF^*$ given by Adjacencies} \label{map}
	
	Having constructed a cobordism, we would now like to construct a map on $HF^*$. Since the Milnor fiber $M_g$ embeds into the Milnor fiber $M_f$, we would like to construct a chain map $CF^*(\check{\phi}_f) \to CF^*(\check{\phi}_g)$ where we view the latter chain complex as a quotient of the former chain complex. In other words, if $[f] \to [g]$, $HF^*$ behaves covariantly when giving us a map $HF^*(\phi^k_f) \to HF^*(\phi^k_g)$.
	
	In order to do this, we need to modify the Milnor fiber of $\tilde{g}$ (which emebds into $M_f$) so that the fixed points of $\phi_g$ are seperated from the fixed points of $\phi_f$. On the other hand, we do not want to change the exact symplectomorphism type of $M_{\tilde{g}}$. Since we may view $M_g \subset M_{\tilde{g}} \subset M_f$, the monodromy $\phi_g$ does not have fixed points in a small tubular neighborhood of $L_g$ as it is compactly supported away from the boundary of $M_g$.
	
	The key idea is inspired by neck-stretching (see, for example, \cite{Wendl}) though we emphasize that this is not neck-stretching in the technical sense of considering a sequence of almost complex structures on longer and longer necks. However, we do glue in something like the symplectization of the link $L_g$. Hence, we'll call this \textbf{neck lengthening}. Since $L_g$ is a contact hypersurface in $M_{\tilde{g}}$ and it separates $M_{\tilde{g}}$, we will neck lengthen the small tubular neighborhood $(-\epsilon,\epsilon) \times L_g$. We view $M_{\tilde{g}} = M_{\tilde{g}}^- \cup_{L_g} M_{\tilde{g}}^+$ where $M_{\tilde{g}}^+$ contains the original boundary of $M_{\tilde{g}}$. A neighborhood of $L_g$ in $(M_{\tilde{g}}, \omega=d\theta)$ can be identified symplectically with $(N_\epsilon, \omega_\epsilon) := ((-\epsilon, \epsilon) \times L_g, d(r\theta) + \omega)$ for sufficiently small $\epsilon > 0$ (different constant from before). We then replace $N_\epsilon$ with larger collars of the form $((-T, T) \times L_g, d (q(r)\theta) + \omega)$, with $C^0$-small function $q$ chosen with $q' > 0$ so that the collar can be glued in smoothly to replace $(N_\epsilon, \omega_\epsilon)$; we want $q'$ to also be $C^1$-small on $(-T+1,T-1)$. This collar (or more commonly, neck) is somewhat like the symplectization of $(L_g, \theta|_{L_g})$. The symplectic manifolds constructed in this way are all exact symplectomorphic. See Figure 4.
	
	\begin{center}
		\tikzset{every picture/.style={line width=0.75pt}} %set default line width to 0.75pt        
		\begin{tikzpicture}[x=0.75pt,y=0.75pt,yscale=-1,xscale=1]
			%Curve Lines [id:da6396681566647202] 
			\draw    (394.33,87.5) .. controls (468.33,86.5) and (448.33,129.5) .. (492.33,129.5) ;
			%Curve Lines [id:da7871289507690344] 
			\draw    (32.33,164.5) .. controls (38.33,194.5) and (80.33,196.5) .. (126.33,196.5) ;
			%Curve Lines [id:da6516390799404674] 
			\draw    (32.33,118.5) .. controls (38.33,88.5) and (79.33,86.5) .. (126.33,87.5) ;
			%Curve Lines [id:da6357990268499514] 
			\draw    (32.33,164.5) .. controls (28.33,154.5) and (29.33,130.5) .. (32.33,118.5) ;
			%Curve Lines [id:da07290102109946894] 
			\draw    (86,148) .. controls (85.33,160.5) and (106.33,160.5) .. (126.33,160.5) ;
			%Curve Lines [id:da22800881145767393] 
			\draw    (86,148) .. controls (85.33,134.5) and (111.33,133.5) .. (126.33,133.5) ;
			%Shape: Ellipse [id:dp8426963087650328] 
			\draw   (485.83,147) .. controls (485.83,137.34) and (488.74,129.5) .. (492.33,129.5) .. controls (495.92,129.5) and (498.83,137.34) .. (498.83,147) .. controls (498.83,156.66) and (495.92,164.5) .. (492.33,164.5) .. controls (488.74,164.5) and (485.83,156.66) .. (485.83,147) -- cycle ;
			%Curve Lines [id:da3258533173053686] 
			\draw    (394.33,196.5) .. controls (468.33,196.5) and (449.33,163.5) .. (492.33,164.5) ;
			%Shape: Ellipse [id:dp049322668669943726] 
			\draw  [dash pattern={on 0.84pt off 2.51pt}] (118.33,110.5) .. controls (118.33,97.8) and (121.92,87.5) .. (126.33,87.5) .. controls (130.75,87.5) and (134.33,97.8) .. (134.33,110.5) .. controls (134.33,123.2) and (130.75,133.5) .. (126.33,133.5) .. controls (121.92,133.5) and (118.33,123.2) .. (118.33,110.5) -- cycle ;
			%Shape: Ellipse [id:dp23389642134858413] 
			\draw  [dash pattern={on 0.84pt off 2.51pt}] (118.83,178.5) .. controls (118.83,168.56) and (122.19,160.5) .. (126.33,160.5) .. controls (130.48,160.5) and (133.83,168.56) .. (133.83,178.5) .. controls (133.83,188.44) and (130.48,196.5) .. (126.33,196.5) .. controls (122.19,196.5) and (118.83,188.44) .. (118.83,178.5) -- cycle ;
			%Straight Lines [id:da036393713860823684] 
			\draw    (126.33,133.5) -- (187.33,133.5) ;
			%Shape: Ellipse [id:dp1961773175290551] 
			\draw  [dash pattern={on 0.84pt off 2.51pt}] (299.33,110.5) .. controls (299.33,97.8) and (302.92,87.5) .. (307.33,87.5) .. controls (311.75,87.5) and (315.33,97.8) .. (315.33,110.5) .. controls (315.33,123.2) and (311.75,133.5) .. (307.33,133.5) .. controls (302.92,133.5) and (299.33,123.2) .. (299.33,110.5) -- cycle ;
			%Shape: Ellipse [id:dp35544726396463] 
			\draw  [dash pattern={on 0.84pt off 2.51pt}] (299.83,178.5) .. controls (299.83,168.56) and (303.19,160.5) .. (307.33,160.5) .. controls (311.48,160.5) and (314.83,168.56) .. (314.83,178.5) .. controls (314.83,188.44) and (311.48,196.5) .. (307.33,196.5) .. controls (303.19,196.5) and (299.83,188.44) .. (299.83,178.5) -- cycle ;
			%Straight Lines [id:da7420578995689722] 
			\draw    (307.33,87.5) -- (394.33,87.5) ;
			%Straight Lines [id:da838151478020321] 
			\draw    (307.33,196.5) -- (394.33,196.5) ;
			%Straight Lines [id:da49058426494640894] 
			\draw    (307.33,133.5) -- (394.33,133.5) ;
			%Straight Lines [id:da1839593071787735] 
			\draw    (307.33,160.5) -- (394.33,160.5) ;
			%Curve Lines [id:da9963552519350567] 
			\draw    (394.33,133.5) .. controls (410.33,132.5) and (412.33,139.5) .. (413.33,148.5) ;
			%Curve Lines [id:da49236706840602573] 
			\draw    (394.33,160.5) .. controls (403.33,158.5) and (413.33,157.5) .. (413.33,148.5) ;
			%Straight Lines [id:da6478256750338136] 
			\draw    (126.33,160.5) -- (187.33,160.5) ;
			%Straight Lines [id:da5138407694331215] 
			\draw    (126.33,87.5) -- (187.33,87.5) ;
			%Straight Lines [id:da800640636860108] 
			\draw    (126.33,196.5) -- (187.33,196.5) ;
			%Straight Lines [id:da7924865458339465] 
			\draw    (246.33,196.5) -- (307.33,196.5) ;
			%Straight Lines [id:da16702591456848537] 
			\draw    (246.33,160.5) -- (307.33,160.5) ;
			%Straight Lines [id:da045931434765150714] 
			\draw    (246.33,133.5) -- (307.33,133.5) ;
			%Straight Lines [id:da2959671540775435] 
			\draw    (246.33,87.5) -- (307.33,87.5) ;
			
			% Text Node
			\draw (48,133.4) node [anchor=north west][inner sep=0.75pt]    {$M_{g}$};
			% Text Node
			\draw (426,133.4) node [anchor=north west][inner sep=0.75pt]    {$M_{f}$};
			% Text Node
			\draw (210,99) node [anchor=north west][inner sep=0.75pt]   [align=left] {...};
			% Text Node
			\draw (211,168) node [anchor=north west][inner sep=0.75pt]   [align=left] {...};
		\end{tikzpicture}
		
		\caption{Figure 4: Neck Lengthening}
	\end{center}
	\vspace{3mm}
	
	Next, we modify the function $-F_{\phi_{\tilde{g}}}$ which defines the action for $\phi_{\tilde{g}}$. Let $A_g$ be the lowest action of fixed points contained in $M^-_{\tilde{g}}$ and $A_f$ be the highest action of fixed points contained in $M^+_{\tilde{g}}$. We then define a smooth function $-\widetilde{F}$ to equal $-F_{\phi_{\tilde{g}}}|_{M_{\tilde{g}}^-} + |A_f - A_g| + 1$ on $M_{\tilde{g}}^-$. We glue in a neck $(-(T+1),T+1) \times L_g$ such that $-\widetilde{F} = -Cr$ on $(-T,T) \times L_g$ for some $C >0$; i.e. is linear. The symplectic structure on the neck is $d(q(r)\theta) +d\theta= q'(r)dr \wedge\theta + (q(r)+1)d\theta$ and on the part where $F=-Cr$ is linear, the Hamiltonian vector field is simply $\pm \frac{C}{q'(r)} \cdot X_R$, where $X_R$ is the Reeb vector field. Below is a schematic picture where we split along the separating hypersurface $L_g$ and add in a cylinder on which there is a function $-\widetilde{F}=-Cr$.
	
	\begin{center}
		\tikzset{every picture/.style={line width=0.75pt}} %set default line width to 0.75pt        
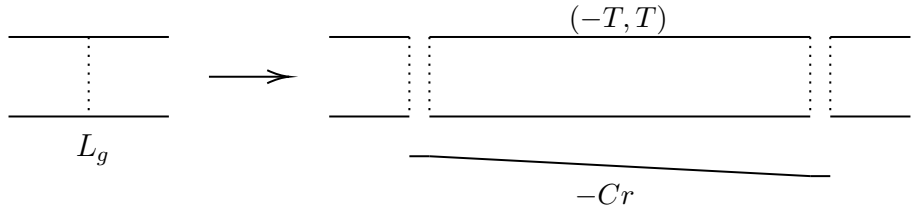
		\begin{tikzpicture}[x=0.75pt,y=0.75pt,yscale=-1,xscale=1]
			
			%Straight Lines [id:da6974828657433427] 
			\draw    (100,70) -- (180,70) ;
			%Straight Lines [id:da25101210936747687] 
			\draw    (100,110) -- (180,110) ;
			%Straight Lines [id:da30604879238378624] 
			\draw  [dash pattern={on 0.84pt off 2.51pt}]  (140,70) -- (140,110) ;
			%Straight Lines [id:da7032589825973263] 
			\draw    (200,90) -- (238,90) ;
			\draw [shift={(240,90)}, rotate = 180] [color={rgb, 255:red, 0; green, 0; blue, 0 }  ][line width=0.75]    (10.93,-3.29) .. controls (6.95,-1.4) and (3.31,-0.3) .. (0,0) .. controls (3.31,0.3) and (6.95,1.4) .. (10.93,3.29)   ;
			%Straight Lines [id:da545638771802325] 
			\draw    (260,70) -- (300,70) ;
			%Straight Lines [id:da005562811122439948] 
			\draw  [dash pattern={on 0.84pt off 2.51pt}]  (300,70) -- (300,110) ;
			%Straight Lines [id:da20309274979717462] 
			\draw    (260,110) -- (300,110) ;
			%Straight Lines [id:da2660503128664067] 
			\draw    (510,110) -- (550,110) ;
			%Straight Lines [id:da4241241909004845] 
			\draw  [dash pattern={on 0.84pt off 2.51pt}]  (510,70) -- (510,110) ;
			%Straight Lines [id:da13165776361545256] 
			\draw    (510,70) -- (550,70) ;
			%Straight Lines [id:da8809451732357261] 
			\draw    (310,130) -- (500,140) ;
			%Straight Lines [id:da8701822600033844] 
			\draw    (300,130) -- (310,130) ;
			%Straight Lines [id:da5690249598570762] 
			\draw    (500,140) -- (510,140) ;
			%Straight Lines [id:da22728641239740122] 
			\draw    (310,70) -- (500,70) ;
			%Straight Lines [id:da9596851054118256] 
			\draw    (310,110) -- (500,110) ;
			%Straight Lines [id:da1607242692161992] 
			\draw  [dash pattern={on 0.84pt off 2.51pt}]  (310,70) -- (310,110) ;
			%Straight Lines [id:da7091132035656531] 
			\draw  [dash pattern={on 0.84pt off 2.51pt}]  (500,70) -- (500,110) ;
			
			% Text Node
			\draw (378,52.4) node [anchor=north west][inner sep=0.75pt]  [font=\small]  {$( -T,T)$};
			% Text Node
			\draw (381,142.4) node [anchor=north west][inner sep=0.75pt]  [font=\small]  {$-Cr$};
			% Text Node
			\draw (132,117.4) node [anchor=north west][inner sep=0.75pt]    {$L_{g}$};
		\end{tikzpicture}
	\end{center}
	
	The constant $C$ is chosen to satisfy the requirement that $-\widetilde{F} = -F_{\phi_{\tilde{g}}}$ restricted to $M_{\tilde{g}}^+$ and that $C/q'(r)$ be smaller than $P$, the period of the smallest nonconstant Reeb orbit (remember $q'(r)>0$).
	
	In $(-(T+1),-T) \times L_g$ and $(T,T+1) \times L_g$, $-\tilde{F}$ is required to have negatives derivative in the $r$ coordinate and to be $C^\infty$-small. This whole construction can be viewed as choosing a Hamiltonian $H$ such that $-\widetilde{F} = -F_{\phi_{\tilde{g}}} + H$. By choosing a large $T$, this separates the actions of orbits of $f$ and of $g$ since we can basically shift the actions of the orbits of $f$ by at least $-2C(T-1)$ and $T$ can be arbitrarily large. Hence, the actions of the fixed points of $\phi_g$ when using $-\widetilde{F}$ are all greater than the actions of the fixed points of $\phi_f$ when $T$ is large. Below is a picture which basically combines the previous mapping torus cobordism picture with this kind of neck lengthening but we've flipped everything over (because we're using $-\widetilde{F}$ rather than $+\widetilde{F}$). The outer torus for the Milnor fiber of $\tilde{g}$ is ``wider'' to show that we've lengthened and the dashed line is meant to show that the actions of the fixed points for $\phi_{\tilde{g}}$ are all higher than before. The region between the two shaded disks is irrelevant in this situation and one could excise it if preferred.
	
	\begin{center}

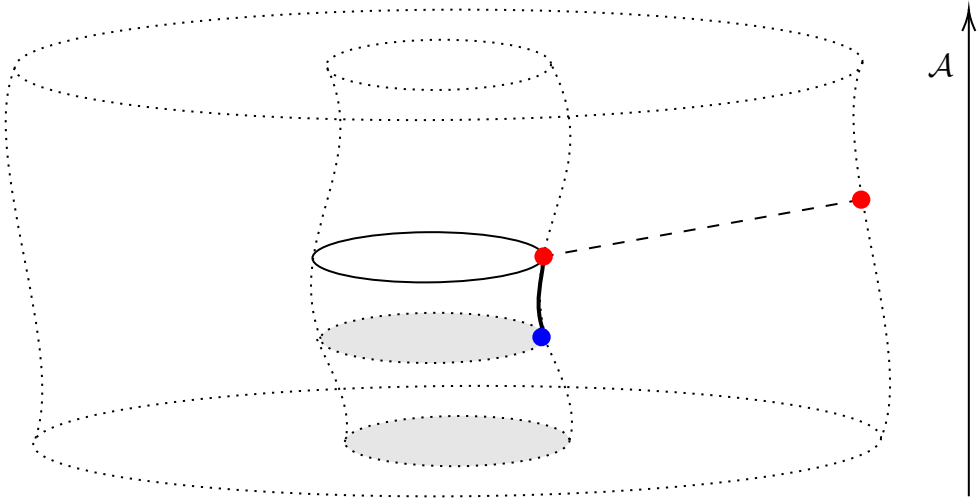
		\tikzset{every picture/.style={line width=0.75pt}} %set default line width to 0.75pt        
		
		\begin{tikzpicture}[x=0.75pt,y=0.75pt,yscale=-1,xscale=1]
			%uncomment if require: \path (0,349); %set diagram left start at 0, and has height of 349
			
			%Curve Lines [id:da03902398664157358] 
			\draw  [dash pattern={on 0.84pt off 2.51pt}]  (392.78,272.06) .. controls (397.66,256.02) and (386.93,236.1) .. (379.81,220.15) ;
			%Straight Lines [id:da5252653174785189] 
			\draw [line width=0.75]  [dash pattern={on 4.5pt off 4.5pt}]  (538.31,150.99) -- (379.85,179.57) ;
			%Curve Lines [id:da20816205512679198] 
			\draw  [dash pattern={on 0.84pt off 2.51pt}]  (538.97,81.65) .. controls (520.29,125.79) and (566.87,227.03) .. (548.19,270.92) ;
			%Shape: Ellipse [id:dp0444441153884918] 
			\draw  [dash pattern={on 0.84pt off 2.51pt}] (538.97,81.65) .. controls (539.08,96.97) and (444.48,110.09) .. (327.68,110.95) .. controls (210.87,111.8) and (116.09,100.07) .. (115.98,84.75) .. controls (115.87,69.42) and (210.46,56.3) .. (327.27,55.45) .. controls (444.07,54.59) and (538.85,66.32) .. (538.97,81.65) -- cycle ;
			%Shape: Ellipse [id:dp9526452298757178] 
			\draw  [fill={rgb, 255:red, 0; green, 0; blue, 0 }  ,fill opacity=0.1 ][dash pattern={on 0.84pt off 2.51pt}] (380.81,220.14) .. controls (380.86,227.12) and (355.8,232.95) .. (324.82,233.18) .. controls (293.85,233.41) and (268.7,227.94) .. (268.65,220.97) .. controls (268.6,213.99) and (293.66,208.16) .. (324.64,207.93) .. controls (355.61,207.7) and (380.76,213.17) .. (380.81,220.14) -- cycle ;
			%Shape: Ellipse [id:dp16383012150859466] 
			\draw  [dash pattern={on 0.84pt off 2.51pt}] (383.55,82.79) .. controls (383.61,89.76) and (358.54,95.6) .. (327.57,95.82) .. controls (296.59,96.05) and (271.44,90.58) .. (271.39,83.61) .. controls (271.34,76.64) and (296.41,70.8) .. (327.38,70.57) .. controls (358.35,70.35) and (383.5,75.81) .. (383.55,82.79) -- cycle ;
			%Curve Lines [id:da13939771510494858] 
			\draw  [dash pattern={on 0.84pt off 2.51pt}]  (380.81,220.14) .. controls (366.39,163.25) and (412.22,138.91) .. (383.55,82.79) ;
			%Curve Lines [id:da1991715281581663] 
			\draw  [dash pattern={on 0.84pt off 2.51pt}]  (268.65,220.97) .. controls (249.36,170.69) and (295.33,132.43) .. (271.39,83.61) ;
			%Shape: Ellipse [id:dp10948276888675768] 
			\draw   (379.85,179.57) .. controls (379.9,186.54) and (354.12,192.38) .. (322.27,192.61) .. controls (290.43,192.85) and (264.57,187.38) .. (264.52,180.41) .. controls (264.47,173.44) and (290.24,167.6) .. (322.09,167.36) .. controls (353.94,167.13) and (379.8,172.59) .. (379.85,179.57) -- cycle ;
			%Curve Lines [id:da3574546404542658] 
			\draw [line width=1.5]    (380.81,220.14) .. controls (373.69,203.2) and (379.6,191.15) .. (379.85,179.57) ;
			%Shape: Circle [id:dp3071105744152871] 
			\draw  [color={rgb, 255:red, 0; green, 0; blue, 255 }  ,draw opacity=1 ][fill={rgb, 255:red, 0; green, 0; blue, 255 }  ,fill opacity=1 ] (382.89,220.13) .. controls (382.91,222.38) and (381.1,224.23) .. (378.84,224.24) .. controls (376.59,224.26) and (374.74,222.44) .. (374.73,220.19) .. controls (374.71,217.93) and (376.53,216.09) .. (378.78,216.08) .. controls (381.04,216.06) and (382.88,217.87) .. (382.89,220.13) -- cycle ;
			%Shape: Circle [id:dp1560319709664637] 
			\draw  [color={rgb, 255:red, 255; green, 0; blue, 0 }  ,draw opacity=1 ][fill={rgb, 255:red, 255; green, 0; blue, 0 }  ,fill opacity=1 ] (383.93,179.54) .. controls (383.95,181.79) and (382.13,183.63) .. (379.88,183.65) .. controls (377.62,183.67) and (375.78,181.85) .. (375.76,179.6) .. controls (375.75,177.34) and (377.56,175.5) .. (379.82,175.48) .. controls (382.07,175.47) and (383.91,177.28) .. (383.93,179.54) -- cycle ;
			%Shape: Circle [id:dp22132568933414576] 
			\draw  [color={rgb, 255:red, 255; green, 0; blue, 0 }  ,draw opacity=1 ][fill={rgb, 255:red, 255; green, 0; blue, 0 }  ,fill opacity=1 ] (542.39,150.96) .. controls (542.41,153.21) and (540.59,155.05) .. (538.34,155.07) .. controls (536.08,155.09) and (534.24,153.27) .. (534.23,151.02) .. controls (534.21,148.76) and (536.02,146.92) .. (538.28,146.9) .. controls (540.53,146.89) and (542.38,148.7) .. (542.39,150.96) -- cycle ;
			%Shape: Ellipse [id:dp553613249773824] 
			\draw  [dash pattern={on 0.84pt off 2.51pt}] (548.19,270.92) .. controls (548.3,286.24) and (453.7,299.36) .. (336.9,300.22) .. controls (220.09,301.07) and (125.31,289.34) .. (125.2,274.02) .. controls (125.09,258.69) and (219.69,245.58) .. (336.49,244.72) .. controls (453.3,243.86) and (548.08,255.59) .. (548.19,270.92) -- cycle ;
			%Curve Lines [id:da28164210780805377] 
			\draw  [dash pattern={on 0.84pt off 2.51pt}]  (115.98,84.75) .. controls (97.3,128.89) and (143.88,230.13) .. (125.2,274.02) ;
			%Shape: Ellipse [id:dp3235544006131379] 
			\draw  [fill={rgb, 255:red, 0; green, 0; blue, 0 }  ,fill opacity=0.1 ][dash pattern={on 0.84pt off 2.51pt}] (392.78,272.06) .. controls (392.83,279.03) and (367.76,284.87) .. (336.79,285.09) .. controls (305.81,285.32) and (280.66,279.85) .. (280.61,272.88) .. controls (280.56,265.91) and (305.63,260.07) .. (336.6,259.84) .. controls (367.57,259.62) and (392.72,265.09) .. (392.78,272.06) -- cycle ;
			%Curve Lines [id:da8259846782311808] 
			\draw  [dash pattern={on 0.84pt off 2.51pt}]  (280.61,272.88) .. controls (285.49,256.84) and (273.76,236.93) .. (266.65,220.98) ;
			%Straight Lines [id:da9170365986031928] 
			\draw    (592,300.22) -- (592,57.22) ;
			\draw [shift={(592,55.22)}, rotate = 90] [color={rgb, 255:red, 0; green, 0; blue, 0 }  ][line width=0.75]    (10.93,-3.29) .. controls (6.95,-1.4) and (3.31,-0.3) .. (0,0) .. controls (3.31,0.3) and (6.95,1.4) .. (10.93,3.29)   ;
			
			% Text Node
			\draw (570,76.4) node [anchor=north west][inner sep=0.75pt]    {$\mathcal{A}$};
		\end{tikzpicture}
	\end{center}
	
	Let $a$ be a value between these two sets of actions. Above, we chose $C$ to be smaller than the period of the smallest periodic Reeb orbit of $\theta|_{L_g}$. In other words, we think of $H$ as giving a Hamiltonian symplectomorphism of small positive slope. Then, we have a perturbation $\check{\phi}_{\tilde{g}}^N$ equal to the composition of $\phi_{\tilde{g}}$ with this $C^\infty$-small Hamiltonian symplectomorphism of small positive slope. The $N$ stands for neck lengthening. This $CF^*(\check{\phi}_{\tilde{g}}^N)$ is isomorphic to $CF^*(\check{\phi}_{\tilde{g}})$ where $\check{\phi}_{\tilde{g}}^N$ is a small positive slope perturbation as in the definition without neck lengthening. This is because neck lengthening does not create or destroy orbits and counts of $J_t$-holomorphic curves are unchanged so long as $J_t$ is cylindrical on the neck.
	
	With this in hand, the action filtration defines a subcomplex $CF^*_{\leq a}(\check{\phi}_{\tilde{g}}^N)$ where the generators have action less than $a$. These generators are precisely those in $M^+_f$ and $CF^*(\check{\phi}_g) \cong CF^*(\check{\phi}_{\tilde{g}}^N)/CF^*_{\leq a}(\check{\phi}_{\tilde{g}}^N)$. 
	
	We may now define a chain map $\Psi:CF^*(\check{\phi}_f) \to CF^*(\check{\phi}_{\tilde{g}}) \cong CF^*(\check{\phi}_{\tilde{g}}^N)/CF^*_{\leq a}(\check{\phi}_{\tilde{g}}^N)$. The idea is to use the cobordism between mapping tori $E$ from before and make a count of trajectories. Then we use an exact symplectomorphism $\chi: M_{\tilde{g}} \to M_{\tilde{g}}^N$ between the Milnor fiber and the neck lengthened Milnor fiber to get an associated isomorphism between mapping tori $T(\phi_{\tilde{g}}) \to T(\chi \phi_{\tilde{g}} \chi^{-1})$. This ensures that the chain map respects the action filtration. Lastly, we quotient by the subcomplex $CF^*_{\leq a}(\check{\phi}_{\tilde{g}}^N)$. Thus, the next step is to clarify how to count trajectories in $E$ and in particular, establish compactness results to ensure finite counts.
	
	Since $E \subset \C^{n+1}$, it has symplectic and almost complex structures. For a generic choice of cylindrical almost complex structure $J$, consider a map $u:\R \times \R/\Z \to E$ satisfying $\partial_s u + J \partial_t u = 0$ and $\lim_{s \to \pm \infty} u(s,t) = \gamma_\pm(t)$ where $\gamma_+$ is a simple Reeb orbit of the monodromy of $\phi_f$--it corresponds to a fixed point $p_+$--and similary $\gamma_-$ is a Reeb orbit of the monodromy of $\phi_g$ corresponding to a fixed point $p_-$. As will be shown below, because of a certain $J$-convex hypersurface in $E$ near the vertical boundary, these cylinders cannot have interior tangencies to it due to the maximum principle. This allows us to apply Gromov compactness and obtain compact 0-dim oriented moduli spaces $\overline{\mathcal{M}}(E,J,p_-,p_+)$. Then $\Psi(p_+) = \sum_{p_-}\#^\pm \overline{\mathcal{M}}(E,J,p_-,p_+) \cdot p_- $ where $(-CZ(p_-))-(-CZ(p_+)) =+1$. Standard Floer theory techniques show that this is a chain map.

	\section{Stable Hamiltonian Structures and Hofer Energy} \label{ham}
	
	In order to obtain the compactness result needed to define the map above, we will dedicate this section to some of the setup needed. In particular, we need to at least define what we mean by finite energy Floer trajectories. For a more complete survey, see ch. 6 of Wendl \cite{Wendl}.
	
	\begin{definition}
		A \textbf{stable Hamiltonian structure} (SHS for short) on an oriented $(2n- 1)$-dimensional manifold $M$ is a pair $(\Lambda,\Omega)$ consisting of 1-form $\Lambda$ and a closed 2-form $\Omega$ such that:
		\begin{enumerate}
			\item $\Omega|_{\ker \Lambda}$ is nondegenerate.
			\item $\ker \Omega \subset \ker d\Lambda$.
		\end{enumerate}
	\end{definition}
	
	Moreover, a stable Hamiltonian structure $(\Lambda,\Omega)$ gives a co-oriented hyperplane distribution $\xi := \ker \Lambda$ and a positively transverse vector field $R$ determined by the conditions $\Omega(R, \cdot) \equiv 0$ and $\Lambda(R) \equiv 1$. This is analogous to contact manifolds and so we'll call $R$ the \textbf{Reeb vector field}. Indeed, contact manifolds are an example of manifolds with stable Hamiltonian structures.
	
	\begin{example}
		Suppose $(M,\alpha)$ is a contact manifold. Then $(\alpha,d\alpha)$ is a stable Hamiltonian structure. The second property is trivially satisfied while the first property is built into the definition for $\alpha$ to be a contact form. $\xi = \ker \alpha$ is the usual hyperplane distribution and $R$ coincides with the usual Reeb vector field of contact geometry.
	\end{example}
	
	\begin{example}
		For isolated hypersurface singularities defined by polynomial $f:\C^{n+1} \to \C$ and $\phi$ is the monodromy, the mapping torus lives in $\C^{n+1}$ and inherits lots of structure. For example, $(\pi^* d\theta, i^*\omega_0)$ as a stable Hamiltonian structure where $i^*\omega_0$ is the restriction of the standard symplectic structure $\omega_0$ on $\C^n$ and $d\theta$ is a closed 1-form on $S^1$. Observe that for tangent vectors $v \in TM_f$, $d\pi(v) = 0$ and hence $v \in \ker \pi^* d\theta$. On the other hand, we have the linear map $df$ (not a differential form since $f$ is not a real function) and $\ker df|_{T_f} = \bigcup_{\theta \in S^1} TM_f(\theta)$ where $M_f(\theta) = f^{-1}(\theta)$. So $\ker \pi^* d\theta = \ker df|_{T_f}$.
		
		On the other hand, the fibers are symplectic submanifolds with respect to $\omega_0$ and so $\omega_0|_{\ker df}$ is nondegenerate. Moreover, $\pi^* d\theta$ is closed, so $\ker d(\pi^*d\theta) = T\C^n$ and hence contains everything, including $\ker \omega_0$.
		
		Note that this stable Hamiltonian structure is not contact since the distribution $\ker df|_{T_f}$ is integrable. But there is a contact structure we could put on $T_f$ if we use the so-called generalized Thurston-Winkelnkemper construction.
		
		For higher iterates of the monodromy, we can use the branched covers discussed earlier to put stable Hamiltonian structures on those mapping tori.
	\end{example}
	
	One useful feature of a SHS on a manifold $M$ is that there is a symplectization. The \textbf{symplectization} of $(M, \Lambda,\Omega)$ for any stable Hamiltonian structure $(\Lambda,\Omega)$ can be defined by choosing suitable diffeomorphisms of $(-\epsilon, \epsilon) \times M$ with $\R \times M$. Equivalently, we may consider $\R \times M$ with the family of symplectic forms $\omega_\psi$ defined by $\omega_\psi:= d(\psi(r)\Lambda) + \Omega$ where $\psi$ is any element of $\mathcal{T} := \{\psi \in C^\infty(\R,(-\epsilon,\epsilon)):\psi' > 0 \}$. Note that $\omega_\psi = \psi' dr \wedge \Lambda + \psi d\Lambda + \Omega$. If we restrict $\omega_\psi$ to $\xi = \ker \Lambda$, then the first term vanishes.
	
	Having a way to define a symplectization in hand, it's natural to consider $J$-holomorphic cylinders in the symplectization. But what sort of $J$ do we consider?
	
	\begin{definition}
		Let $r$ be the coordinate on $(-\epsilon,\epsilon)$. Then $J$ is an \textbf{admissible} almost complex structure if
		\begin{enumerate}
			\item $J$ is invariant under translation in the $(-\epsilon,\epsilon)$ factor.
			\item $J \partial_r = R$ (Reeb vector field)
			\item $J(\xi) = \xi$ (the hyperplane distribution is $J$-invariant)
			\item $J|_\xi$ is compatible with the symplectic vector bundle structure $\Omega|_\xi$.
		\end{enumerate}
		
		An alternative nomenclature is to call such $J$ \textbf{cylindrical} since they bear much similarity to the contact case.
	\end{definition}
	
	These properties are enough to show that an admissible $J$ is tamed by every $\omega_\psi$. Thus, we may define the \textbf{Hofer energy} of a $J$-holomorphic curve $u:(\Sigma,j) \to (\R \times M,J)$ where $J$ is admissible:
	
	\[E(u) := \sup_{\psi \in \mathcal{T}} \int_\Sigma u^*\omega_\psi.\]
	
	One can show that each $\omega_\psi$ tames an admissible $J$ and hence, $E(u) \geq 0$ with equality if and only if $u$ is a constant map. One important remark that Wendl points out is that this notion of energy is different from some other notions that appear in symplectic geometry, such as Hofer energy. However, for the purposes of getting uniform bounds in order to have compactness of moduli spaces, this notion of energy is sufficient.
	
	\textbf{Note:} One reason to define the Hofer energy in this way is for the following reason. Consider a contact manifold $(M,\alpha)$ and a Reeb orbit $\gamma$ of period $T> 0$. A \textbf{trivial cylinder} $u: \R \times S^1 \to (\R \times M, d(e^t\alpha))$ is $u(s,t) = (Ts,\gamma(Tt))$ and if we use $\int_{\R \times S^1} u^*d(e^t \alpha)$ as the energy, the energy would be
	\[\lim_{s \to + \infty} \int_{S^1} u^* e^t \alpha - \lim_{s \to -\infty} \int_{S^1} u^* e^t \alpha = \infty.\]
	
	The point of choosing $\psi:\R \to (-\epsilon,\epsilon)$ is get a finite quantity (without changing the symplectomorphism type); we then take a supremum since there is no canonical choice of $\psi$.
	
	In the mapping torus case, we have that $\omega_\psi = d(\psi(r)\pi^* d\theta + i^* \alpha_0)$ where $\alpha_0$ is a primitive 1-form for $\omega_0$; say $\alpha_0 = \frac{1}{2}\sum y_kdx_k - x_k dy_k$. Then, if we look at finite energy curves $u:\Sigma \to \R \times T_f$, they have to limit to positive and negative orbits. We'll denote the positive and negative ends with $\Gamma^\pm$. So the energy is
	
	\[E(u):=\sum_{\gamma_+ \in \Gamma^+} \lim_{r \to +\infty} \int_{S^1} \gamma^*_+ (\psi(r) \pi^* d\theta +i^*\alpha_0) - \sum_{\gamma_- \in \Gamma^-} \lim_{r \to -\infty} \int_{S^1} \gamma^*_- (\psi(r) \pi^* d\theta +i^*\alpha_0).\]
	
	Supposing $\psi \to \pm \epsilon$ as $r \to \pm \infty$, the energy becomes
	
	\[E(u):= \epsilon\sum_{\gamma_+ \in \Gamma^+}  \int_{S^1} \gamma^*_+ (\pi^* d\theta +i^*\alpha_0) - \epsilon\sum_{\gamma_- \in \Gamma^-} \int_{S^1} \gamma^*_- (\pi^* d\theta +i^*\alpha_0).\]
	
	The $\gamma^* \pi^* d\theta$ gives the winding number whereas $\gamma^* i^* \alpha_0 = \gamma^* \alpha_0$ (since the image of $\gamma$ is in $T_f$) gives the action or length of the orbit. If the orbit is near the horizontal boundary of $T_f$, it will have very large action. At least, on a Milnor fiber which is Stein, as we go to infinity, the volume increases exponentially. Since the mapping torus, near infinity, is a circles worth of Stein manifolds, the energy is increasing.
	
	\section{Compactness of Moduli of Floer Trajectories} \label{cpt}
	
	We're now prepared to prove the following:
	
	\begin{theorem}
		Let $p_-,p_+$ be two fixed points of the monodromy symplectomorphism with $(-CZ(p_-))-(-CZ(p_+)) =1$. Then, the 0-dim oriented moduli space $\overline{\mathcal{M}}(E,J,p_-,p_+)$ of finite energy Floer trajectories that limit to the Reeb orbits corresponding to $p_-$ and $p_+$ is compact and hence, finite.
	\end{theorem}
	
	\begin{proof}
		To get compactness results, we first note that in the cobordism between mapping tori $T_f$ and $T_g$ and also in the symplectizations, because the monodromy is compactly supported, there is a hypersurface of the form $\R \times L \times S^1$ where $L$ is the link and $L \times S^1 \subset T_f$ (again, the monodromy being compactly supported means that the mapping torus is trivial near the boundary). Note that since $L \subset T_f$ is a contact submanifold, $\R \times L$ is a symplectic submanifold of $\R \times T_f$ and that $\R\times L \times S^1$ is a union of all the slices of $\R \times L$ transported by the Reeb flow. Let $\psi$ be the Stein Morse function on a Milnor fiber and let $\Psi$ be the trivial extension of it to the mapping torus $T_f$ near the boundary (again, the boundary is trivial because the monodromy is compactly supported). Then in a neighborhood of the boundary, the level sets of $\Psi$ are $L \times S^1$. Taking $pr_2:\R \times T_f \to T_f$, the map $\Psi \circ pr_2$ has $\R \times L \times S^1$ as a regular level set and $dd^c (\Psi \circ pr_2) \leq 0$ since $dd^c \psi \leq 0$. Hence, $\R \times L \times S^1$ is a $J_0$-convex hypersurface. We now state a lemma (from Oancea's survey, Lemma 1.4, \cite{Oancea}).
		
		\begin{lemma}
			Let $S \subset M$ be a $J$-convex hypersurface and $\psi$ a (local) function of definition. No $J$-holomorphic curve $u: (D^2(0, 1), i) \to M$ can have an interior strict tangency point with $S$; i.e. $\psi \circ u$ cannot have a strict local maximum.
		\end{lemma}
		
		$\R \times L \times S^1$ is a $J_0$-convex hypersurface with $\Psi \circ pr_2$ as the function of definition. Since the trajectories converge to Reeb orbits lying in a compact set, any trajectory has its ends contained in the compact set. This means that a sequence of trajectories that escape towards the boundary will give a trajectory with an interior tangency point. This contradicts the lemma and hence, we cannot actually have a sequence of trajectories escaping. In other words, the image of the Floer trajectories of finite energy limiting to the given Reeb orbits must all lie in a compact set in the target. Then the two conditions of Gromov compactness are fulfilled: finite energy and images lying in a compact set. We conclude that the moduli space is compact.
	\end{proof}
	
	\section{(1+1) TQFT Equipped with Partial Lefschetz Fibrations} \label{tqft}
	
	Having shown that the maps we defined for the concrete cobordisms are sensible because we have compact moduli spaces, we can then ask various questions about their properties. For example, do the maps depend on the open book decomposition or perhaps only on data that comes from our polynomial $f$? In this section, we'll develop a bit of theory to address this question but provide more than we need because this theory may be of independent interest. In some of the discussion here, we will speak of Floer homology as opposed to \textit{co}homology but there is really no difference. We discuss homology because in \cite{SeidelDehn} (which we borrow heavily from), Seidel presents Floer \textit{homology} of a symplectomorphism $\phi:M \to M$ alternatively as the Floer
	homology of the mapping torus $M \times [0,1]/\sim$ where $(x,t) \sim
	(\phi(x),t+1)$. This is a natural fibration over an oriented circle $Z$ (to
	keep with his notation). However, the formalism presented works for any
	symplectic fibration $F \to Z$. Moreover, when we have a connected, compact,
	oriented surface $S$ with $p+q$ boundary circles divided into $p$ positive and
	$q$ negative ends, this induces maps on tensor products of the homology. In
	more detail, let $\partial S = \bar{Z}^-_1 \cup...\cup \bar{Z}^-_p \cup
	Z^+_{p+1} \cup ...\cup Z^+_{p+q}$. Then, given a symplectic fibration $E \to S$
	with fiber $M$, the restrictions $F^\pm_k = E|_{Z^\pm_k}$ are symplectic
	fibrations and we have a relative Gromov invariant
	
	\[G(S,E):\bigotimes^p_{k=1} HF_*(Z^-_k,F^-_k) \longrightarrow \bigotimes^{p+q}_{k=p+1} HF_*(Z^+_k,F^+_k).\]
	
	\begin{remark}
		In Seidel's lecture notes, he sometimes assumes that $H^1(M,\R) = 0$ in order to use the $C^\infty$-topology on $\text{Symp}(M,\omega)$. For us, when $n \geq 3$, the homotopy type of the Milnor fiber $M_f$ of $f:\C^{n+1} \to \C$ is that of $\bigvee^\mu S^n$ where $\mu$ is the Milnor number. Hence, $H^1(M_f,\R)=0$. When $n=1$, the fiber is finitely many points and when $n=2$, the fiber is a curve, a well-studied situation.
	\end{remark}
	
	\noindent For our purposes, we mostly care about $S = S^1 \times [0,1]$ with a
	positive and negative end. However, we'll like to extend this formalism to
	include maps between fibrations with different fibers. In order to do this, we
	need to extend the theory to that of partial Lefschetz fibrations (see McLean's
	preprint \cite{McLeanGrowth}). 
	
	\begin{definition}
		A \textbf{partial Lefschetz fibration} is comprised of a quadruple $(E,S,\pi, K)$ where $E$
		is a smooth manifold with corners, $K \subset \text{Int}(E)$ is a compact subset in the interior of $E$, $\pi:E \setminus K \to S$ is a map between manifolds. Moreover, $E$ consists of two codimension 1 boundary components $\partial_h E$ (horizontal boundary) and $\partial_v E$ (vertical boundary) meeting in a codimension 2 component. There is a 1-form $\theta_E$ on $E$ making $E$ into a Liouville domain after smoothing the corners. The map $\pi$ must satisfy the following properties:
		
		\begin{enumerate}
			\item A neighborhood of $\partial_h E$ is diffeomorphic to $S \times (1-\epsilon, 1] \times \partial F$
			where $F$ is some Liouville domain called the \textbf{fiber} of $\pi$. Here $\theta_E = \theta_S + r \alpha_F$ where $\theta_S$ is a Liouville form on $S$ and $r$ parameterizes the interval. The 1-form $\alpha_F$ is the contact form on $\partial F$. The map $\pi$ is the projection map to $S$ in this neighborhood.
			
			\item $\theta_E$ restricted to the fibers of $\pi$ is non-degenerate away from the singularities of $\pi$.
			
			\item The restriction $\pi|_{\partial_v E}$ is a fibration whose fibers are exact symplectomorphic to $F$ such that the fibers of $\pi$ are either disjoint from $\partial_V E$ or entirely contained in $\partial_v E$.
			
			\item There are only finitely many singularities of $\pi$ and they are all disjoint from the boundary $\partial E$. They are modeled on non-degenerate holomorphic singularities.
		\end{enumerate}
	\end{definition}
	
	The most general theory would not require a global Liouville form on $E \setminus K$ but only when restricted to the fibers. The way to think about such a definition is perhaps to first begin with $K = \varnothing$ and $S = D \subset \C$ which reduces the situation to the usual Lefschetz fibrations that the reader may be accustomed to. In all situations, the preimage $\pi^{-1}(\partial S)$ is the vertical boundary but in particular when $S$ is the unit disk, $\pi|_{\partial_v E}:\partial_v E \to S^1$ is part of the data of an open book decomposition.
	
	When the compact set $K \neq \varnothing$, one might think of it as ``hiding'' the critical points of any extension of $\pi$ to $E$; these critical points can be quite pathological and nonisolated, which allows us to consider, for example, cobordisms between mapping tori with different fibers. These cobordisms can be treated with classical Morse theory where we think of handle attachment as occuring when we traverse pass a critical value. Lefschetz theory, being over $\C$, cannot recover this since one can always go around a critical value. In fact, we can take $K$ to be larger so that $\pi$ has no singularities at all.
	
	Also, near the boundary $\partial_v E$, we have a connection given by the $\omega_E$-orthogonal plane field to the fibers. Because the fibration is a product near $\partial_h E$, the parallel transport maps associated to this connection are well defined and are compactly supported if we transport around a loop. The symplectomorphism $\phi: F \to F$ given by parallel transporting around a loop on $\partial S$ is called, without surprise, the monodromy symplectomorphism around this boundary component.
	
	\subsection{Unique Analytic Continuation $\Rightarrow$ $HF^*$ Depends Only On the Framed Binding}
	
	This theory may seem unsuitable for studying holomorphic sections since
	such curves seem uncontrolled within $K$. However, this is not the case. Let
	\[\mathcal{M}_{E,K} = \{u:S \to E:\bar{\partial}_J u = 0, \pi \circ (u|_{u^{-1}(E \setminus K)}):u^{-1}(E \setminus K) \to S \text{ is the natural injection} \} \]
	
	\noindent When a holomorphic curve is injective on an open set (and in fact,
	the identity), then there is a unique analytic continuation of the curve due to the identity theorem: a holomorphic map is completely determined by its restriction to open sets and in the case of curves, determined even just by a sequence of points with an accumulation point. Hence, the moduli space above does not depend on $K$.
	
	Within this formalism, we're able to show that our maps do not depend on the
	open book decomposition of the contact manifold (the mapping tori) but only on
	the binding and its normal framing. Recall that $W$ is the section of the Milnor fiber of $\tilde{g}$ which has links $L_g$ and $L_f$ as boundary.
	
	\begin{lemma} Let $M \to Z$ be an open book decomposition over an oriented
		circle $Z$ with page $F$ being a Liouville domain and $B$ the codim 2 binding. Then, up to isomorphism, $HF_*(Z,M)$ only depends on $B$ and a choice of normal framing.
	\end{lemma}
	
	\begin{proof} Let the total space $M$ be given two open book decompositions
		with pages $F_1,F_2$ and the same binding $B$ and trivialization of the
		normal bundle. We then have two trivial cobordisms on the total space
		$M$, call them $E_{12}$ and $E_{21}$. Note that this does not affect
		the orientation of $M$. In each of these, we can choose compact sets
		$K_{12},K_{21}$ such that their complements are basically small
		neighborhoods of the boundary of each $E_{12},E_{21}$. For each
		situation, the boundary is a union of vertical and horizontal
		components. The horizontal boundary is, under the trivialization,
		$[0,1] \times \partial F_i$, $i=1,2$. But $\partial F_1 = \partial F_2
		= B$.
		
		Thus, we have two partial Lefschetz fibrations. The gluing of
		$E_{12}$ and $E_{21}$ gives a trivial cobordism between $M$ and itself
		equipped with the same open book. This induces the identity map which
		means that the composition of the maps induced by $E_{12}$ and $E_{21}$
		must be inverses. The same is true if we glue in the opposite way:
		$E_{21}$ to $E_{12}$. Hence, each induced map is an isomorphism and $HF_*(Z,M)$ does not depend on the choice of open book decomposition, so long as we fix the binding and choice of
		normal framing.
	\end{proof}
	
	The argument for this lemma is another proof of some of the results of Appendix B \cite{McLeanLog}. Moreover, it implies the following:
	
	\begin{corollary} Let $M_f,M_{\tilde{g}}$ be the mapping tori from before
		which fiber over $Z_f,Z_{\tilde{g}}$, respectively. Then, if $E \to A$
		is the cobordism over the annulus from Section \ref{cobordism}, using this notation, we have an induced map $\Psi$ of Section \ref{map} which passes to cohomology $HF^*(Z_f,M_f) \to HF^*(Z_{\tilde{g}},M_{\tilde{g}})$. This $\Psi$ does not depend on the choice of open books but only on the horizontal boundary of $E$ which is given by $W \times A$, under the canonical normal framing.
	\end{corollary}
	
	This means that the map really only ultimately depends on information that is supplied by the algebro-geometric data.
	
	\section{Novel Proof of Zariski's Conjecture} \label{sunflower}
	
	Now that we've established some properties of this map, let's apply it towards addressing a conjecture of Zariski. Recall that the Milnor number of an analytic hypersurface singularity represented by $f$ can be algebraically defined as $\mu:=\dim_\C \mathcal{O}/\text{Jac}(f)$. It is a fact that the singularity is isolated if and only if $\mu < \infty$. In the work of \cite{BobadillaPelka}, they prove results which imply the following: 
	
	\begin{theorem}
		If a family of isolated hypersurface singularities has constant Milnor number $\mu$, then they also have constant multiplicity. 
	\end{theorem}	
	
	This is a specialized form of Zariski's conjecture \cite{Zariski} which states: If $f,g:(\C^{n+1},0) \to (\C,0)$ define singularities with the same topological type; i.e. there is a homeomorphism $\Phi: (\C^{n+1},0) \to (\C^{n+1},0)$ such that $\Phi(V_f) = V_g$, then the singularities have the same multiplicity. In the case that the singularities are isolated, Milnor proved that the Milnor number is a topological invariant and hence, de Bobadilla-Pe\l ka's work proves Zariski's conjecture for families of isolated hypersurface singularities. However, as of 2022, the conjecture is open for pairs of singularities that do not fit into a family and is also open for general classes of nonisolated singularities though Massey pointed out that their work applies to certain types of nonisolated singularities.

	As mentioned in the introduction, the de Bobadilla-Pe\l ka proof uses log resolutions and builds something called an A'Campo space for it which involves some tropical geometry. They also use McLean's work on fixed-point Floer cohomology for Milnor fibrations and multiplicity. The main technical result is an extension of McLean's spectral sequence to include the data of fixed points at infinity.
	
	In this section, we provide an alternative and somewhat simpler proof (which relies on McLean's original spectral sequence). Consider polynomials $f,g:\C^{n+1} \to \C$ with isolated singularities and suppose that they fit into an adjacent family. We can define the Milnor fibration for $f$ using radius $\delta$ circle and cutoff the total space with an $\epsilon$ ball. For the Milnor fibration of $g$, use constants $\delta',\epsilon'$ where these are much smaller than their respective counterparts for $f$. Then, the Milnor fiber $M_g$ embeds into the Milnor fiber of $M_f$ and the extra part is a cobordism $W = M_f \setminus M_g$. Here's the picture from before:
	
	\begin{center}
		\tikzset{every picture/.style={line width=0.75pt}} %set default line width to 0.75pt
		
		\begin{tikzpicture}[x=0.75pt,y=0.75pt,yscale=-1,xscale=1]
			%uncomment if require: \path (0,372); %set diagram left start at 0, and has height of 372
			
			%Curve Lines [id:da5420547952716865]
			\draw [line width=1.5]    (355.42,224.86) .. controls (367.33,257.5) and (414.77,206.09) .. (436.57,262.62) ;
			%Curve Lines [id:da24025337540987635]
			\draw [line width=1.5]    (355.54,111.81) .. controls (363.33,86.5) and (375.81,89.52) .. (385.38,89.76) .. controls (394.95,90.01) and (423.33,113.17) .. (433.77,64.84) ;
			%Curve Lines [id:da6739804758408963]
			\draw    (349.22,211.33) .. controls (371.4,256.44) and (381.54,279.12) .. (391.33,301.5) ;
			%Curve Lines [id:da21397622713626885]
			\draw    (316.83,168.1) .. controls (388.19,164.14) and (422.57,211.62) .. (436.57,262.62) ;
			%Curve Lines [id:da8731683817333415]
			\draw    (316.83,168.1) .. controls (396.58,147.36) and (413.33,116.5) .. (432.77,70.4) ;
			%Shape: Ellipse [id:dp1581194322078936]
			\draw  [dash pattern={on 0.84pt off 2.51pt}] (164.33,168.1) .. controls (164.33,83.91) and (232.61,15.66) .. (316.83,15.66) .. controls (401.06,15.66) and (469.33,83.91) .. (469.33,168.1) .. controls (469.33,252.29) and (401.06,320.54) .. (316.83,320.54) .. controls (232.61,320.54) and (164.33,252.29) .. (164.33,168.1) -- cycle ;
			%Shape: Ellipse [id:dp6585130852464707]
			\draw  [dash pattern={on 0.84pt off 2.51pt}] (248.63,168.1) .. controls (248.63,130.45) and (279.16,99.92) .. (316.83,99.92) .. controls (354.5,99.92) and (385.04,130.45) .. (385.04,168.1) .. controls (385.04,205.75) and (354.5,236.28) .. (316.83,236.28) .. controls (279.16,236.28) and (248.63,205.75) .. (248.63,168.1) -- cycle ;
			%Shape: Ellipse [id:dp5335788850778409]
			\draw  [color={rgb, 255:red, 255; green, 0; blue, 0 }  ,draw opacity=1 ][fill={rgb, 255:red, 255; green, 0; blue, 0 }  ,fill opacity=1 ] (350.64,224.86) .. controls (350.64,222.22) and (352.78,220.08) .. (355.42,220.08) .. controls (358.06,220.08) and (360.2,222.22) .. (360.2,224.86) .. controls (360.2,227.5) and (358.06,229.64) .. (355.42,229.64) .. controls (352.78,229.64) and (350.64,227.5) .. (350.64,224.86) -- cycle ;
			%Shape: Ellipse [id:dp13121358083732826]
			\draw  [color={rgb, 255:red, 0; green, 0; blue, 255 }  ,draw opacity=1 ][fill={rgb, 255:red, 0; green, 0; blue, 255 }  ,fill opacity=1 ] (428.99,69.62) .. controls (428.99,66.98) and (431.13,64.84) .. (433.77,64.84) .. controls (436.41,64.84) and (438.55,66.98) .. (438.55,69.62) .. controls (438.55,72.26) and (436.41,74.4) .. (433.77,74.4) .. controls (431.13,74.4) and (428.99,72.26) .. (428.99,69.62) -- cycle ;
			%Curve Lines [id:da5880687424034943]
			\draw    (232.33,135.5) .. controls (255.33,106.5) and (324.78,159.72) .. (349.22,211.33) ;
			%Curve Lines [id:da043699630475921625]
			\draw    (234.33,202.5) .. controls (263.33,228.5) and (326.48,171.98) .. (351.35,119.39) ;
			%Curve Lines [id:da9402739527339292]
			\draw    (351.35,119.39) .. controls (365.8,87.22) and (372.94,78.25) .. (388.33,33.5) ;
			%Shape: Ellipse [id:dp4548857614142561]
			\draw  [color={rgb, 255:red, 255; green, 0; blue, 0 }  ,draw opacity=1 ][fill={rgb, 255:red, 255; green, 0; blue, 0 }  ,fill opacity=1 ] (350.76,111.81) .. controls (350.76,109.17) and (352.9,107.03) .. (355.54,107.03) .. controls (358.18,107.03) and (360.32,109.17) .. (360.32,111.81) .. controls (360.32,114.45) and (358.18,116.59) .. (355.54,116.59) .. controls (352.9,116.59) and (350.76,114.45) .. (350.76,111.81) -- cycle ;
			%Shape: Ellipse [id:dp15238140945376166]
			\draw  [color={rgb, 255:red, 0; green, 0; blue, 255 }  ,draw opacity=1 ][fill={rgb, 255:red, 0; green, 0; blue, 255 }  ,fill opacity=1 ] (431.79,262.62) .. controls (431.79,259.98) and (433.93,257.84) .. (436.57,257.84) .. controls (439.21,257.84) and (441.35,259.98) .. (441.35,262.62) .. controls (441.35,265.26) and (439.21,267.4) .. (436.57,267.4) .. controls (433.93,267.4) and (431.79,265.26) .. (431.79,262.62) -- cycle ;
			%Curve Lines [id:da02062722520202498]
			\draw [line width=0.75]  [dash pattern={on 0.84pt off 2.51pt}]  (364.33,118.17) .. controls (374.58,103.53) and (381.81,102.52) .. (391.38,102.76) .. controls (400.95,103.01) and (430.33,117.17) .. (439.77,77.84) ;
			%Curve Lines [id:da26112729605040985]
			\draw [line width=0.75]  [dash pattern={on 0.84pt off 2.51pt}]  (345.54,105.03) .. controls (364.33,64.5) and (377.76,76.26) .. (387.33,76.5) .. controls (396.91,76.74) and (409.89,95.49) .. (424.33,61.17) ;
			%Curve Lines [id:da5499778128623531]
			\draw [line width=0.75]  [dash pattern={on 0.84pt off 2.51pt}]  (362.42,218.86) .. controls (379.33,243.5) and (420.77,199.09) .. (442.57,255.62) ;
			%Curve Lines [id:da7305752068209586]
			\draw [line width=0.75]  [dash pattern={on 0.84pt off 2.51pt}]  (345.42,231.86) .. controls (368.33,269.5) and (420.33,223.17) .. (429.33,271.17) ;
			%Curve Lines [id:da29975410013770265]
			\draw    (234.33,202.5) .. controls (216.33,189.5) and (218.33,155.5) .. (232.33,135.5) ;
			%Shape: Circle [id:dp011268652518372502]
			\draw  [fill={rgb, 255:red, 0; green, 0; blue, 0 }  ,fill opacity=1 ] (315.58,167.85) .. controls (315.58,166.61) and (316.59,165.6) .. (317.83,165.6) .. controls (319.08,165.6) and (320.08,166.61) .. (320.08,167.85) .. controls (320.08,169.09) and (319.08,170.1) .. (317.83,170.1) .. controls (316.59,170.1) and (315.58,169.09) .. (315.58,167.85) -- cycle ;
		\end{tikzpicture}
		
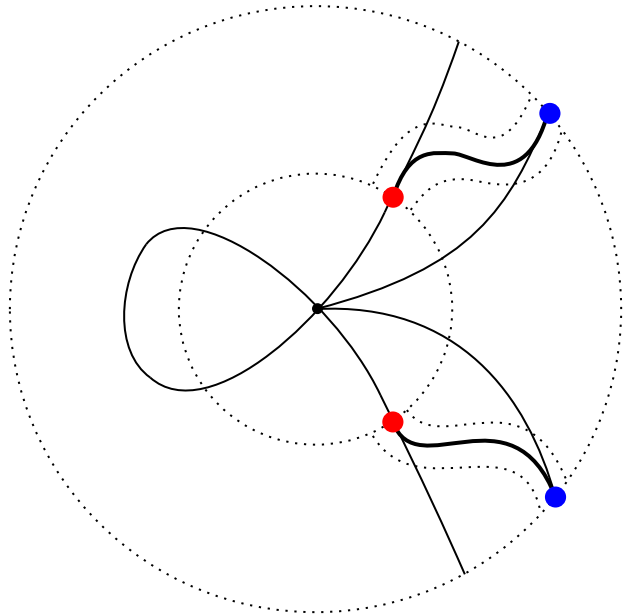
		\caption{Figure: A cobordism between Milnor fibrations}
	\end{center} 
	
	In this picture, there's a modification of $g$ to some $\tilde{g}$ so that it agrees with $f$ near the boundary of $B_\epsilon$. There may be some extra topology appearing when we enlarge the ball but the radius $\delta'$ monodromy $\phi^k_g$ won't do anything to the extra part. However, if we enlarge both the ball and the circle for the monodromy of $g$ to radii $\epsilon$ and $\delta$, respectively, then it may see extra singularities. In the picture below, when we perturb $f$ to $g$, some critical value(s) split off (depicted in red) and the monodromy $\phi^k_g$ does not see them.
	
	\begin{center}
		\tikzset{every picture/.style={line width=0.75pt}} %set default line width to 0.75pt        
		
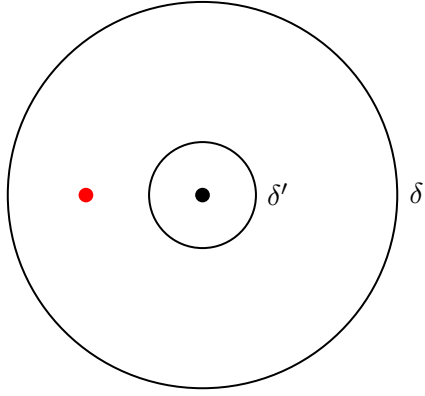
		\begin{tikzpicture}[x=0.75pt,y=0.75pt,yscale=-1,xscale=1]
			
			%Shape: Circle [id:dp6814100145032034] 
			\draw   (182,124.17) .. controls (182,70.5) and (225.5,27) .. (279.17,27) .. controls (332.83,27) and (376.33,70.5) .. (376.33,124.17) .. controls (376.33,177.83) and (332.83,221.33) .. (279.17,221.33) .. controls (225.5,221.33) and (182,177.83) .. (182,124.17) -- cycle ;
			%Shape: Circle [id:dp4725210038425216] 
			\draw   (252.5,124.17) .. controls (252.5,109.44) and (264.44,97.5) .. (279.17,97.5) .. controls (293.89,97.5) and (305.83,109.44) .. (305.83,124.17) .. controls (305.83,138.89) and (293.89,150.83) .. (279.17,150.83) .. controls (264.44,150.83) and (252.5,138.89) .. (252.5,124.17) -- cycle ;
			%Shape: Circle [id:dp38257959146413856] 
			\draw  [fill={rgb, 255:red, 0; green, 0; blue, 0 }  ,fill opacity=1 ] (276,124.17) .. controls (276,122.42) and (277.42,121) .. (279.17,121) .. controls (280.92,121) and (282.33,122.42) .. (282.33,124.17) .. controls (282.33,125.92) and (280.92,127.33) .. (279.17,127.33) .. controls (277.42,127.33) and (276,125.92) .. (276,124.17) -- cycle ;
			%Shape: Circle [id:dp07768379954423565] 
			\draw  [color={rgb, 255:red, 255; green, 0; blue, 0 }  ,draw opacity=1 ][fill={rgb, 255:red, 255; green, 0; blue, 0 }  ,fill opacity=1 ] (217.83,124.19) .. controls (217.82,122.44) and (219.23,121.01) .. (220.98,121) .. controls (222.73,120.99) and (224.16,122.4) .. (224.17,124.15) .. controls (224.18,125.9) and (222.77,127.32) .. (221.02,127.33) .. controls (219.27,127.34) and (217.84,125.93) .. (217.83,124.19) -- cycle ;
			
			% Text Node
			\draw (310,116.55) node [anchor=north west][inner sep=0.75pt]  [font=\small]  {$\delta '$};
			% Text Node
			\draw (381,116.55) node [anchor=north west][inner sep=0.75pt]  [font=\small]  {$\delta $};	
		\end{tikzpicture}	
	\end{center}
	
	But if we enlarge the circle to radius $\delta$ and consider a new monodromy, denote it as $\psi^k_g$, then it will take into account the extra critical values and have some new and nontrivial action on $B_\epsilon$. Note that if we used a smaller cutoff $B_{\epsilon'}$; i.e. restrict $\psi^k_g$ to a smaller ball, then $\psi^k_g|_{B_{\epsilon'}} = \phi^k_g$ because we've taken away the piece on which it would perform its new monodromy. So to \textbf{recap}: if we only enlarge the ball, we see extra topology but the monodromy is trivial on the extra part. If we only enlarge the circle, the monodromy will encircle more critical values but not do anything on a restricted ball of small radius. However, if we enlarge both, then there is both extra topology and extra critical values so then the monodromy $\psi^k_g$ is truly different.
	
	\begin{lemma} \label{triangle}
		Let $W = g^{-1}(\delta) \cap (B_\epsilon \setminus B_{\epsilon'})$; this is the extra part of the Milnor fiber when we enlarge the ball. Then there is an exact triangle given by action filtration of orbits:
		\begin{center}
			\begin{tikzcd}
				HF^*(\psi^k_g) \arrow{rr}  & & HF^*(\phi^k_g) \arrow{ld} \\
				& H^*(W,\partial W) \arrow{lu}
			\end{tikzcd}
		\end{center}
	\end{lemma}
	
	\begin{proof}
		Above in Section \ref{map}, we made an action filtering/lengthening the cylinder argument to define the chain map. The discussion there is easily adapted to our current situation basically verbatim to show how to use the action and lengthening to separate out the fixed points of $\psi^k_g$ which are not also fixed points of $\phi^k_g$; let $b$ be a value in between these two sets of actions. However, unlike that situation, we do have extra fixed points due to enlarging the radius of the circle to take into account extra critical values which are seen by the monodromy $\psi^k_g$. These fixed points live in $W$, are made to have much higher action by the lengthening argument, and compute the relative cohomology: $H^*(W,\partial W)$. There is an injective chain map $C^*(W,\partial W) \hookrightarrow CF^*(\psi^k_g)$.
		
		Next, we define a chain map on the subcomplex $CF^*(\psi^k_g)_{\leq b} \to CF^*(\phi^k_g)$ in a similar way as in Section \ref{map} via a count of Floer trajectories. We then extend the map by zero to the full complex $CF^*(\psi^k_g)$; it is clear that the kernel is exactly $C^*(W,\partial W)$. Lastly, the map from $CF^*(\phi^k_g) \to C^*(W,\partial W)$ is just the zero map. This is because the geometric generators for each complex (fixed points or Morse critical points, for example) live in disjoint subsets and moreover, the actions of the generators are also disjointed. Hence, we get an exact triangle but it in fact splits up into a series of short exact sequences.
	\end{proof}
	
	Now, L\^e-Ramanujam \cite{le_ramanujam} (or see \cite{BobadillaPelka}, Prop. 5.22), proved: 
	
	\begin{theorem}[L\^e-Ramanujam]
		Given a family of isolated singularities $f_t:\C^{n+1} \to \C$ with $n \neq 2$ such that the Milnor number for these is constant (independent of the parameter $t$), then the diffeotype of the Milnor fibrations for the $f_t$ is also independent of $t$. 
	\end{theorem}
	
	Their proof relies on the h-cobordism theorem and so when $n=2$, the homeotype, thought not diffeotype, is also independent of $t$ by the work of Freedman. What are the consequences of this theorem?
	
	Given a $\mu$-constant family, the members of the family are adjacent to each other. When we perturb one to the other, there will not be any critical values shooting off because the number of those which can shoot off is bounded above by the difference in Milnor number which, in this case, is zero. So the extended monodromy $\psi^k_g$ will not see any extra critical values and hence, not change. This is because the monodromy on a smoothing of the singular fiber is only nontrivial in the region very near the singularity. So $\psi^k_g$ extends $\phi^k_g$ by the identity map.
	
	What about the extra piece that comes from enlarging to $B_\epsilon$? L\^e-Ramanujam's result tells us that no extra topology appears so this $W$ is smoothly trivial when $n \neq 2$: $W = [0,1] \times L_g$ where $L_g$ is the link of $g$ and boundary of $M_g$. Again, when $n=2$, it is at least topologically trivial. 
	
	\begin{remark} \label{quibble}
		We emphasize that it's possible that there are values of $r \in [\epsilon',\epsilon]$ such that $g^{-1}(\delta)$ does not intersect $S^{2n+1}_r$ transversally which allows for the possibility of being nontrivial as a Stein cobordism. If all the critical points of the Liouville vector field are subcritical, we can perform handlemoves to cancel the points and obtain something in the same Weinstein class. But if there are critical points (and the cobordism isn't flexible), then it may very well have some interesting symplectic topology that the monodromy maps completely overlooks.
	\end{remark}
	
	By Lemma \ref{triangle}, we already know that $HF^*(\phi^k_g) \cong HF^*(\psi^k_g)$. We now want to show that $HF^*(\psi^k_g) \cong HF^*(\phi^k_f)$. The indirect way to do this is to take the branched mapping torus (embedded in some $k$-fold branch cover of $\C^{n+1}$) of each and observe that it is graded contactomorphic to the abstract mapping torus given by a page and self-map: $(M_g(\delta,\epsilon),\psi^k_g)$. This is the content of Giroux's work; see Remark \ref{giroux} in Appendix A.
	
	This abstract mapping torus, in turn, is graded contactomorphic to the boundary of the Lefschetz fibration we get when we Morsify $g$ which is an open book decomposition; by Morsify, we mean that we take a very small perturbation of $g$ to make it a Morse function. The reason for this is due to Picard-Lefschetz theory which tells us that the monodromy is a composition of generalized Dehn twists along the vanishing cycles (the Lagrangian spheres). The Morsification of $f$ is essentially the same as the Morsification of $g$ and hence, gives the same Lefschetz fibration (up to deformation equivalence) if we use the same $\delta$ and $\epsilon$. If we use the larger ones, then Remark \ref{quibble} says there may be some symplectic subtleties that go undetected. However, fixed point Floer cohomology is blind to those differences and is even insensitive to the deformation equivalence because it is only sensitive to the monodromies. So as far as $HF^*$ is concerned, the mapping tori for $\psi^k_g$ and $\phi^k_f$ use the same $\delta$ and $\epsilon$ by comparing them to the boundary of the ``same'' Lefschetz fibration.
	
	The final step is to see that since L\^e-Ramanujam theorem (plus the $n=2$ case) implies that $W \cong_{\text{homeo}} [0,1] \times L_g$ which means that $H^*(W,\partial W) = 0$ for any $n$. Therefore, our exact triangle above produces for us a symplectically-constructed isomorphism: $HF^*(\phi^k_f) \cong HF^*(\phi^k_g)$ for every $k$. By \cite{McLeanLog}, we are able to recover the multiplicity of $f$ simply by looking for the smallest $k$ such that $HF^*(\phi^k_f) \neq 0$; this was discussed in Section \ref{mult}. Moreover, we also recover the log canonical threshold of $f$ by the formula:
	\[\text{lct}_0(f) = \liminf_{k\to \infty}  \left(\inf \left\{-\frac{\alpha}{2k}: HF^\alpha(\phi^k_f,+) \neq 0 \text{ or } -\frac{\alpha}{2k}=1  \right\} \right)\]
	
	\noindent To summarize, we have proven Theorem \ref{zar} which can be phrased more colloquially as:
	
	\begin{theorem} 
		If a family of isolated hypersurface singularities is $\mu$-constant, the multiplicity and log canonical threshold are also constant in the family as a result of symplectic considerations.
	\end{theorem}
	
	\begin{remark}
		We highlight the fact that the Milnor number is a smooth topological invariant and the multiplicity and log canonical threshold are invariants for graded contactomorphic pairs by the work of McLean \cite{McLeanLog}. By comparison, the theorem here is about families with constant Milnor number and we show that the family has the property that each member has the same multiplicity and log canonical threshold and moreover, this is a \emph{symplectic property}. The result was previously proven for the log canonical threshold by Varchenko \cite{varchenko} in 1982 and for multiplicity in 2022 by de Bobadilla-Pe\l ka \cite{BobadillaPelka}. So our work is a new proof and also illustrates that these are not just algebro-geometric properties but symplectic properties.	
	\end{remark}

\begin{remark}
We may also ask about the symplectic invariants for the Milnor fibers of such families. In particular, the Fukaya-Seidel category is a symplectic invariant. In our situation, it is an $A_\infty$ category generated by Lefschetz thimbles that are associated to an ordered set of paths that are assigned to a Morsification of $f$. We won't give the definition here but perhaps one thing to point out is that this category is not the ``full'' Fukaya category since there certainly are more exact Lagrangians than just the thimbles. For example, Keating constructed exact Lagrangian tori for the Milnor fibers of singularities with positive modality \cite{keating} which is in contrast to the result of Ritter that the only exact Lagrangians in Milnor fibers for modality zero singularities (ADE type) of real dimension 4 are spheres \cite{ritter}. At any rate, the Lefschetz fibrations for two adjacent isolated singularities with the same Milnor number are deformation equivalent and hence, they have equivalent Fukaya-Seidel categories.
	
\end{remark}
	
	\section{Follow-up Questions Suggested by Theorem \ref{zar}}
	
	In this section, we pose some questions related to the proof of Theorem \ref{zar} and also discuss related notions and examples.
	
	\begin{question}
		If $f,g$ belong in a $\mu$-constant family, do they have graded contactomorphic links? If so, then this would immediately imply the theorem above. But this is either false or so far unknown.
	\end{question}	
		
		\begin{question}
		There is a graded ring structure on $\bigoplus_{k \geq 0} HF^*(\phi^k,+)$ where $\phi^0=\id$ and the product is the pair-of-pants product $HF^*(\phi^k,+)\otimes HF^*(\phi^\ell,+) \to HF^*(\phi^{k+\ell},+)$. Can this structure be utilized to prove algebraic results such as detecting changes of multiplicity or log canonical threshold within a family?
	\end{question}

	\section{Appendix A: Gradings and the Conley-Zehnder Index for Fixed-Point Floer Cohomology}
	
	Much of this reiterates what is found in \cite{McLeanLog} and \cite{SeidelGraded}; we claim no originality here, only wishing to provide a more self-contained treatment. The point of gradings is that we would otherwise only have a relative grading on $HF^*$ (or the $E^1$ page of the spectral sequence); i.e. we would have to make a choice of what counts as $HF^0$.
	
	What we say here is less detailed. Let $(\R^{2n},\omega_{st})$ be the standard symplectic $\R^{2n}$ and $Sp(2n)$ the group of linear symplectomorphisms. Let $p:\widetilde{Sp}(2n) \to Sp(2n)$ denote its universal cover, recalling that $\pi_1(Sp(2n)) = \Z$. If we have a symplectic vector bundle $\pi:(E,\omega_E) \to X$ of rank $2n$, then we can form the symplectic frame bundle $Fr(E) \to V$ whose fiber is isomorphic to $Sp(2n)$. The fiber over $x \in V$ is the group of linear symplectomorphisms between $(\R^{2n},\omega_{st})$ and $(\pi^{-1}(x),\omega_E|_x)$.
	
	\begin{definition}
		A \textbf{grading} on a bundle $\pi:E \to X$ is a principal $\widetilde{Sp}(2n)$-bundle $\widetilde{Fr}(E)\to X$ together with an isomorphism of principal $Sp(2n)$-bundles:
		\[\iota:\widetilde{Fr}(E) \times_{\widetilde{Sp}(2n)} Sp(2n) \cong Fr(E). \]
	\end{definition}
	
	\noindent Recall that $\widetilde{Fr}(E) \times_{\widetilde{Sp}(2n)} Sp(2n)$ is a quotient $\widetilde{Fr}(E) \times Sp(2n)/ \sim$ where $(x \cdot g, h) \sim (x,p(g)h)$. For example, if $g \in p^{-1}(1) \cong \Z$, then the multiplication $p(g)h = h$ and hence $(x\cdot g, h) \sim (x,h)$.
	
	\begin{example}
		A grading of a symplectic manifold $(X,\omega)$ is a grading of $(TX,\omega) \to X$.
	\end{example}
	
	Now suppose we have the following commutative diagram:
	\begin{center}
		\begin{tikzcd}
			(E_1,\omega_1) \arrow{r}{F} \arrow{d}{\pi_2 }& (E_2,\omega_2) \arrow{d}{\pi_2} \\
			X_1 \arrow{r}{f} & X_2
		\end{tikzcd}
	\end{center}
	
	\noindent where $\pi_i:(E_i,\omega_i) \to X_i$ are symplectic vector bundles of rank $2n$ and $F$ is a map which covers a diffeomorphism $f$ and is fiberwise a linear symplectomorphism. Then there is an induced map on frame bundles $Fr(F):Fr(E_1) \to Fr(E_2)$ defined fiberwise as follows. Let $x \in X_1$. Then for an element $A \in Sp(2n)_x$ in the fiber over $x$, we may compose it with $F_x:E_1|_x \to E_2|_{f(x)}$. This gives an element $F_x \circ A \in Sp(2n)_{f(x)}$. So $Fr(F)$ sends $A \mapsto F_x \circ A$.
	
	A grading of $F$ is a map $\widetilde{F}:\widetilde{Fr}(E_1) \to \widetilde{Fr}(E_2)$ which covers $Fr(F):Fr(E_1) \to Fr(E_2)$. Note that if we have a symplectomorphism $\phi:(X_1,\omega_1) \to (X_2,\omega_2)$, then we can take $d\phi:(TX_1,\omega_1) \to (TX_2,\omega_2)$ to be our $F$ which covers $\phi$. So it is in this sense that we say $\phi$ is graded.
	
	We also want to consider gradings of co-oriented contact manifolds. Suppose $(C,\xi)$ is co-oriented; i.e. $TC/\xi \cong \R$ is oriented. Choose a 1-form $\alpha$ such that $\ker \alpha = \xi$ and $\alpha > 0$ on $TC/\xi$. Then $(\xi,d\alpha|_\xi)$ is a symplectic vector bundle which can be given a grading. Up to isotopy, this grading does not depend on the choice of $\alpha$ so we may sensible define this to be a grading for the co-oriented contact manifold $(C,\xi)$. A grading on a co-orientation preserving contactomorphism $\Psi:(C_1,\xi_1) \to (C_2,\xi_2)$ is a grading on $d\Psi|_{\xi_1}:(\xi_1,d\alpha_1) \to (\xi_2,d\alpha_2)$. For more on isotopies, see \cite{McLeanLog}.
	
	Next, suppose that $B \subset C$ is a contact submanifold of codimension 2. The normal bundle $N_C B := TC|_B/TB \to B$ is symplectic and there is a natural bundle isomorphism $N_CB \cong T^\perp B:= \{v \in \xi: d\alpha(v,w) = 0, \forall w \in TB \cap \xi\}$. In this case where $B$ is codim 2, the tuple $(B \subset C, \xi, \Phi)$ is a contact pair and $\Phi: N_CB \to B \times \C$ is an isomorphism. We call this tupple a \textbf{contact pair with normal bundle data}. A grading for such a pair is a grading in the sense above on $C \setminus B$.
	
	A contactomorphism between two such triples
	$(B_1 \subset C_1, \xi_1, \Phi_1) , (B_2 \subset C_2, \xi_2, \Phi_2)$ is a contactomorphism $\Psi: C_1 \to C_2$ sending $B_1$ to $B_2$ so that the composition
	\begin{center}
		\begin{tikzcd}
			N_{C_1} B_1 \arrow{r}{d\Psi|_{B_1}} & N_{C_2} B_2 \arrow{r}{\Phi_2} & B_2 \times \C \arrow{rr}{(\Psi|_{B_1})^{-1} \times \id_\C} & & B_1 \times \C
		\end{tikzcd}
	\end{center}			
	\noindent is homotopic through symplectic bundle trivializations to $\Phi_1$. If we have such a $\Psi$, we may additionally consider a grading on its restriction $\Psi|_{C_1 \setminus B_1}: C_1 \setminus B_1 \to C_2 \setminus B_2$.
	
	The definition for an abstract contact open book $(M,\theta,\phi)$ is in \cite{McLeanLog} and was mentioned in Section \ref{alg}. To briefly recall, $(M,\theta)$ is a Liouville domain and $\phi:M \to M$ is an exact symplectomorphism supported away from the boundary $\partial M$. From an abstract contact open book, we can construct the mapping torus $T(\phi)$ and then obtain a contact open book: $C_\phi := (\partial M \times D(\delta)) \sqcup T(\phi)/ \sim$ where we glue in a small thickening $\partial M \times D(\delta)$ to the mapping torus. As mentioned above, this is the generalized Thurston-Winkelnkemper construction.
	
	We can give $(M,d\theta)$ a grading as a symplectic manifold and also $\phi$ as a symplectomorphism. Suppose that $(M,\theta,\phi)$ is graded. Then this means we have an isomorphism
	\[\iota:\widetilde{Fr}(E) \times_{\widetilde{Sp}(2n)} Sp(2n) \cong Fr(E). \]
	
	\begin{remark} \label{giroux}
		Briefly, we remark that we may then define what it means for graded abstract contact open books to be isotopic as well as what it means for graded contact open books to be isotopic. One can prove that if two graded abstract contact open books are isotopic, then their associated graded contact open books will also be isotopic. This gives us a map
		\[\{ \text{(graded) abstract contact open books}\}/\text{isotopy} \longrightarrow \{\text{(graded) open books}\}/\text{isotopy}\]
		
		\noindent which was shown to be a bijection by Giroux \cite{Giroux}.
	\end{remark}
	
	Returning to gradings, if we identify a fiber $\widetilde{Fr}(TM)_p \cong \widetilde{Sp}(2n)$, then the grading $\widetilde{\phi}:\widetilde{Fr}(TM)_p \to \widetilde{Fr}(TM)_{\phi(p)}$ can be viewed as a map $\widetilde{\phi}: \widetilde{Sp}(2n) \to \widetilde{Sp}(2n)$. Elements of $\widetilde{Sp}(2n)$ are equivalence classes of paths starting at a basepoint where the equivalence is $\alpha \sim \beta$ if and only if $\alpha,\beta$ have the same endpoints and $\alpha * -\beta$ is a contractible loop. Here, the $*$ means concatenation. Then, $\widetilde{\phi}(\id)$ is a path in $Sp(2n)$. The Conley-Zehnder index, which we explain below, assigns a number to any such path so we will take $CZ(\widetilde{\phi}(\id))$ to define the grading.
	
	For any path of symplectic matrices $(A_t)_{t\in [a,b]}$ we can assign an index $CZ(A_t)$ called its Conley-Zehnder index. The Conley-Zehnder index was originally defined only for certain paths of symplectic matrices $A_t$ but now it's been done in general in \cite{RobbinSalamon1} or \cite{Gutt}. We will not define it here but only list some of its properties (see \cite{Gutt}, Prop. 8):
	
	\begin{enumerate}
		\item $CZ((e^{it})_{t\in [0,2\pi]})=2$.
		\item $(-1)^{n-CZ((A_t)_{t\in[0,1]})} = \text{sign} \det_\R(\id-A_1)$ for any path of symplectic matricies $(A_t)_{t\in[0,1]}$.
		\item $CZ(A_t \oplus B_t) = CZ(A_t) + CZ(B_t)$.
		\item  The Conley-Zehnder index of the catenation of two paths is the sum of their Conley-Zehnder indices.
		\item If $A_t$ and $B_t$ are two paths of symplectic matrices which are homotopic relative to
		their endpoints then they have the same Conley-Zehnder index. Also the index only depends on the path up to orientation-preserving reparameterization.
	\end{enumerate}
	
	\section{Appendix B: Examples of $\mu$-Constant Families of Isolated Singularities} \label{mu}

This appendix is meant to serve as a resource for those less familiar with singularities. All three examples are families previously appearing in the literature where the Milnor number is constant but some other invariants are not, showing that the families are nontrivial. The members of the first family are distinguished by a complex-analytic invariant, the second by a topological invariant, and the third by contact geometry.

\begin{example}
	We are grateful to Jason Starr for reminding us of this first example. Consider the union of 4 lines in $\C^2$, say $xy(x-y)(x-ty)=0$ where $t \in \C \setminus \{0,1\}$ This has the same topological type as the union of four lines $x^4-y^4 = (x+y)(x-y)(x+iy)(x-iy)$. The Milnor number is a topological invariant and while it's cumbersome to compute the Milnor number for the family above, we can easily see from $x^4-y^4$ that $\mu = 9$.
	
	Next, if we take the line $x+y=2$ in $\C^2$, we can compute the points of intersection with the 4 lines; they are $(0,2),(1,1),(2,0),(\frac{2t}{t+1},\frac{2}{t+1})$. We can then use this $x+y=2$ line and some change of coordinates to treat these 4 collinear points as $-1,0,1,\frac{t-1}{t+1}$ in $\C$. The cross-ratio of the four points is therefore $\frac{t-1}{t}$ (a minor calculation) and of course, $t \neq 0,1$ from our conditions above since we didn't want the 4 lines to degenerate to 3 lines with one of them having multiplicity (it should be said that $t \neq \infty$ as well).
	
	If we now compactify to $\mathbb{CP}^2$, there is the following theorem of a more classical flavor:
	\begin{theorem}
		Let $E$ be an arbitrary elliptic curve in $\mathbb{CP}^2$ and $p \in E$. Then there are exactly four lines passing through $p$ which are each tangent to $E$. Conversely, given four lines that all pass through a common point, there is a unique elliptic curve lying tangent to the four prescribed lines.
	\end{theorem}
	
	There is also the following result, likely known to Riemann and Weierstrass:
	
	\begin{proposition}
		Every elliptic curve $(E,0)$ embeds into $\mathbb{CP}^2$.
	\end{proposition}
	
	Now, intersect the four lines from the theorem with a fifth line (not through the marked point) and let $\lambda$ be the cross ratio of the four intersection points. Then the $j$-invariant of the elliptic curve is \[j = 256 \frac{(\lambda^2-\lambda+1)^3}{\lambda^2(\lambda-1)^2}. \]
	
	The $j$-invariant is an intrinsic analytic invariant for elliptic curves and if $\lambda = \frac{t-1}{t}$, the algebra works out so that $j$-invariant is actually of the same form:
	\[j = 256 \frac{(t^2-t+1)^3}{t^2(t-1)^2}. \]
	
	This means that the analytic type of the singularities in this family are distinguished by their cross-ratio which is in fact, \textit{intrinsic}, due to this reformulation with the $j$-invariant. 
	
	\begin{proof}(Sketch of proposition)
		Let $L \to E$ be a line bundle and $K$ the canonical bundle on $E$. Riemann-Roch says:
		\[h^0(E,L)-h^0(X,L^{-1} \otimes K) = \deg(L) + 1-g.\]
		
		Since $g=1$ and $K$ is trivial for elliptic curves, we see that $h^0(E,L) \geq \deg(L)$. So if we want a section with vanishing of order 2 or 3 near the marked point $0 \in E$, such sections exist. Call these sections $x,y$. We use them to define an embedding into $\mathbb{CP}^2$; first we use $x,y$ as coordinates to show that the elliptic curve satisfies the following Weierstrass normal form equation $y^2=x^3+ax^2+bx+c$ and then we projectivize the equation. Note, by the way, that the $y^2$ term on the left tells us there is an involution on the curve. Of course, this is consistent with the group law on $E$ telling us that we have an involution $x \mapsto -x$.
	\end{proof}
	
	\begin{remark}
		One thing that may be of secondary interest is that picking a point on an elliptic curve is like choosing the identity element for the group law and the tangent lines give us the 2-torsion points (on $\R^2/\Z^2$ for example, there are four 2-torsion points: $(0,0), (0,1/2),(1/2,0)$, and $(1/2,1/2)$). In Weierstrass normal form, the inverse to a point $(x,y)$ on the curve is $(x,-y)$ and a 2-torsion point then has $y = 0$. Hence, we're looking for the three roots of $x^3+ax^2+bx+c$; the fourth 2-torsion point is at infinity and is the identity. In any abelian group, the subset of 2-torsion points forms a subgroup and here, it is $\Z/2 \times \Z/2$. The sum of any two of these points is equal to a third in the group which is obvious from the fact that the three roots are collinear.
	\end{remark}	
	
	Also, the example of $x^4+y^4$ is discussed on the level of directed Fukaya categories by Keating \cite{keating_vid}.
	
	Another example of constant Milnor number and variation of moduli: take a family of elliptic surface singularities such as $x^3+y^3+z^3-txyz=0$, where $t^3 \neq 27$ (without this condition, the singularities are nonisolated). Of course this is related to cross-ratio, since this affine equation in $\C^3$ is just the affine cone over the elliptic curve with the same equation in $\mathbb{CP}^2$.

\end{example}

	\begin{example}
		An \textbf{unfolding} of a singularity defined by any polynomial $f$ can be viewed as a deformation space. For example, the miniversal deformation space is an unfolding. If $S$ is a semi-universal unfolding (I don't know what the semi-universal means but $S$ can be viewed as a ball inside some $\C$-vector space), then $f_s = f + \sum s_j \phi_j$ is a deformation of $f$. For each $s$, we can consider $V(f_s)$, the zero locus of $f_s$. The space $D \subset S$ is comprised of those $s$ such that $V(f_s)$ has singularities. If $H_0 \subset S$ is a generic 2-plane passing through the origin, then $H_0 \cap D$ is a curve called a \textbf{generic discriminant curve}. If $H$ is a parallel 2-plane to $H_0$, then $H \cap D$ is called a \textbf{generic unfolded discriminant curve}.
		
		For plane curves defined by $f(x,y)=0$, we have the Milnor number $\mu$ as well as an analytic invariant $\sigma := 1 + \dim \C[x,y]/I$ where $I$ is the ideal generated by $f$, its 1st partial derivatives, and its 2nd partial derivatives.
		
		\noindent \textbf{Fact:} Suppose $f = y^3 - P(x)y+Q(x)$. Then the topological type of $H_0 \cap D$ is determined by $\mu$ and $\sigma$. This fact implies that the topological type of the unfolded discriminant curve $H \cap D$ only depends on $\mu$ and $\sigma$.
		
		Now for Pham's example, originally appearing in \cite{pham}; another reference is \cite{lonne}. Let $f(x,y)=y^3+x^k+1$ and choose $m \geq 2$ such that $2k \leq 3m-2$ and $m \leq k$. He showed there is a 3-parameter unfolding $F(x,y;u,t,s)$ so that for $u=t=0$, $\mu$ is constant. For a fixed sufficiently small $s_0$, the discriminant curve of $F(x,y,;u,t,s_0)$, denoted $D_{s_0} :=\{s=s_0\} \cap D \subset \C^2_{u,t}$, is reduced (all irreducible components have multiplicity 1) and is topologically equivalent to $D_0$. Observe that our 2-plane $\{s=s_0\}$ is not generic. On the other hand, $\sigma = k$ for $s=0$ and $\sigma = m$ for sufficiently small $s \neq 0$.
		
		Since the topological type of the (unfolded) generic discriminant curve is determined by $\mu$ and $\sigma$, this $F$ gives us examples where $\mu$ is constant but the topological type of the generic discriminant curve changes.
	\end{example}
	
	\begin{example}
		Let $f_p = z_0^p + \sum_{j=1}^{2m+1} z_j^2$ be a weighted homogeneous polynomial with $p \equiv \pm 1 \pmod{8}$ and $m \geq 1$. It's Milnor number is $p-1$. Brieskorn \cite{brieskorn} proved that the link of this polynomial, denoted $\Sigma(p,2,...,2)$ is diffeomorphic to the standard $S^{4m+1}$.
		
		In \cite{ustilovsky}, Ustilovsky showed that for each $p$, we get a contact structure $\xi_p$ on $S^{4m+1}$ and these are all pairwise nonisomorphic, distinguished by contact homology. In fact, for each homotopy class of almost contact structure on $S^{4m+1}$, there are infinitely many pairwise nonisomorphic contact structures.
		
		This is in contrast to the result of Caubel-Nemethi-Popescu-Pampu \cite{cnp} which says that if we fix a smooth, oriented 3-manifold $Y$, then it admits at most one contact structure $\xi$ which is Milnor fillable; i.e. $(Y,\xi)$ is contactomorphic to the link $L$ of some isolated surface singularity. In fact, if we are only interested in $Y=S^3$, then Mumford \cite{mumford} showed that $S^3$ can only arise as the link of a smooth point. He does so by showing how to compute $\pi_1(L)$ given the singularity and that $\pi_1(L) = 0$ if and only if the link is of a smooth point and therefore, must be $S^3$. We remark that Mumford's result can be viewed as a proof of the Poincar\'e conjecture in the restricted setting of links of isolated surface singularities.
	\end{example}

	\addcontentsline{toc}{section}{\numberline{}References}

	\bibliographystyle{alpha}
	\bibliography{ref.bib}
\end{document}